\documentclass[a4paper,10pt,notitlepage]{article}

\title{Parameter estimation in a subcritical percolation model with colouring}

\author{Felix Beck\footnote{Centre for Biological Systems Analysis (ZBSA), University of Freiburg, Habsburgerstra\ss{}e~49, 79104~Freiburg, Germany} \footnote{Institute for Mathematics, University of Freiburg, Eckerstra\ss{}e 1, 79104~Freiburg, Germany}, Bence M\'elyk\'uti\textsuperscript{$\ast$}\footnote{Corresponding author. Email: \texttt{melykuti@stochastik.uni-freiburg.de}.}}

\usepackage{amssymb} 
\usepackage{amsmath}
\usepackage[amsmath,thmmarks]{ntheorem}
\usepackage[utf8]{inputenc} 
\usepackage[T1]{fontenc}
\usepackage[UKenglish]{babel}
\reversemarginpar
\usepackage{graphicx}
\usepackage{array} 

\usepackage[sort&compress]{natbib}
\setcitestyle{numbers,square,comma}
\usepackage{hyperref}

\theorembodyfont{\normalfont}

\newtheorem{thm}{Theorem}
\newtheorem{lem}[thm]{Lemma}
\newtheorem{prop}[thm]{Proposition}

\theoremstyle{nonumberplain}

\theoremstyle{nonumberplain}
\theoremsymbol{\ensuremath{\square}}
\newtheorem{proof}{Proof}

\newcommand{\E}{\mathrm{E}}
\renewcommand{\P}{\mathrm{P}}
\newcommand{\diag}{\mathrm{diag}}

\newcommand{\T}{^{\mathrm{T}}}
\newcommand{\argmin}[1]{{\displaystyle \mathop{\mathrm{arg\,min}}_{#1}}}

\newcommand{\conv}[1]{{\displaystyle \mathop{\longrightarrow}_{#1}}}

\begin{document}

\maketitle

\begin{abstract}

In the bond percolation model on a lattice, we colour vertices with $n_c$ colours independently at random according to Bernoulli distributions. A vertex can receive multiple colours and each of these colours is individually observable. The colours colour the entire component into which they fall. Our goal is to estimate the $n_c +1$ parameters of the model: the probabilities of colouring of single vertices and the probability with which an edge is open. The input data is the configuration of colours once the complete components have been coloured, without the information which vertices were originally coloured or which edges are open.

We use a Monte Carlo method, the method of simulated moments to achieve this goal. We prove that this method is a strongly consistent estimator by proving a uniform strong law of large numbers for the vertices' weakly dependent colour values. We evaluate the method in computer tests. The motivating application is cross-contamination rate estimation for digital PCR in lab-on-a-chip microfluidic devices.\\


\end{abstract}

\section{Bond percolation with colouring}\label{s:intro}

We consider bond percolation~\cite{Grimmett_1999} on the triangular lattice, but our arguments hold for the square lattice as well. The vertex set of the infinite lattice is denoted by~$L$. Edges are open (that is, included in the graph, alternatively, receive weight $1$ as opposed to~$0$) independently at random with probability~$\mu\in[0,1]$. There are $n_c\in\mathbb{N}\setminus\{0\}$ colours given, and for every colour $\ell\in\{1,2,\dots,n_c\}$, a parameter $\lambda^\ell\in[0,1]$ is fixed. For every vertex~$i\in L$, the vertex is coloured with colour $\ell\in\{1,2,\dots,n_c\}$ according to a Bernoulli random variable with probability $\lambda^\ell$. The colouring with different colours is independent in any one vertex, and it is also independent among different vertices. A vertex can receive multiple colours and each of these colours is individually observable. We call this colouring the \emph{seeding}: $X^\ell_i\in\{0,1\}$ for every $i\in L$ and $\ell\in\{1,2,\dots,n_c\}$.

These colours propagate through open edges and colour (\emph{`contaminate'}) the entire component they are contained in. Let $i\leftrightarrow j$ mean that vertices $i,j\in L$ are connected by an open path. The observed colour configuration is
\[Y^\ell_i:=X^\ell_i\vee\bigvee_{\substack{j\in L\\j\leftrightarrow i}} X^\ell_j \in\{0,1\}\]
for every $i\in L$ and $\ell\in\{1,2,\dots,n_c\}$, where $\vee$ is the maximum operator.

We also consider this process on finite, connected subsets of the lattice, $I\subset L$. (Here connected is meant with all lattice edges considered, not only the open edges.) Picking the vertex set $I$ implicitly fixes its edge set, the edges which connect vertices of~$I$. We let $n_I:=|I|$. We write $i\sim j$ for adjacent lattice vertices $i,j\in L$ no matter in what state the connecting edge is.

Often we consider nested sequences of such $I$ where each successor is a superset of its predecessor and $n_I\to\infty$. We fix an ordering of the vertices of the infinite lattice~$L$ which is compatible with this sequence as $n_I\to\infty$, that is, each $I$ comprises vertices labelled with $\{1, \dots, n_I\}$. We use $I_2:=\{(i,j)\in I\times I\ |\ i\sim j,\ i<j\}$ for the set of ordered pairs of adjacent vertices (independently of whether the connecting edge is open or closed) and $n_p:=|I_2|$ for the total number of possible edges within~$I$. We define the \emph{exterior vertex boundary} of a subset~$I$ by
\[\Delta I:=\{j\in L\ |\ j\notin I,\ \exists i\in I:\ i\sim j\}.\]
We always require that in our sequences, $|\Delta I|/|I|\to 0$ and for the triangular lattice, $n_p \sim 3 n_I$ (asymptotic equality; $n_p \sim 2 n_I$ is the corresponding condition for the square lattice).


For a fixed~$I$, we define a variant of~$Y^\ell_i$ that is determined exclusively by the seeding and edges in~$I$:
\[\widetilde{Y}^\ell_i:=X^\ell_i\vee\bigvee_{\substack{j\in I\\j\widetilde{\leftrightarrow} i}} X^\ell_j \in\{0,1\}\]
for every $i\in I$ and $\ell\in\{1,2,\dots,n_c\}$. Here $\widetilde{\leftrightarrow}$ means connectedness by open edges in the edge set of~$I$.

Our goal is to estimate the parameter $\theta=(\lambda^1,\dots,\lambda^{n_c},\mu)$ from the data $\big(\widetilde{Y}^\ell_i\big)_{i\in I,\ell\in\{1,2,\dots,n_c\}}$ (Figure~\ref{f:process}). The spatial arrangement of~$\big(\widetilde{Y}^\ell_i\big)$ within the lattice is known, but the seeding~$(X^\ell_i)$ and the open or closed state of the edges are unavailable. We bring together four theoretical tools in this paper.

First, parameter estimation is conducted by the \emph{method of simulated moments} (MSM) \cite{Gourieroux_Monfort_1991, Gourieroux_Monfort_1996} (Section~\ref{s:MSM}). This is a simulation-based, computationally intensive statistical method that yields a point estimate for $\theta$ which converges almost surely to the correct value as $n_I\to\infty$.

Second, as the first step towards proving the strong consistence of the estimator, we prove a strong law of large numbers (SLLN) with weakly dependent variables. We do this in Section~\ref{s:SLLN} by adapting Theorem~1 of~\cite{Etemadi_1983a}.

Third, the SLLN result requires some grasp of how small the dependence is between distant vertices of the lattice. The upper bounds on correlations are provided by the FKG and BK inequalities of percolation theory and the exponential decay of the cluster size distribution \cite{Fortuin_Kasteleyn_Ginibre_1971, vandenBerg_Kesten_1985, Aizenman_Newman_1984}, \cite[Chapters~2 and~6]{Grimmett_1999} in Section~\ref{s:corr}.

Fourth, for the strong consistence of the estimator, we extend the SLLN to be uniform in the parameter vector. We verify in Section~\ref{s:ULLN} that the conditions of a sufficient condition for the uniform law of large numbers (ULLN) hold \cite[p.~8, 2~Theorem]{Pollard_1984} \cite[p.~25, Lemma~3.1]{vandeGeer_2010}.

Our estimation method is tested on synthetic data with known parameter values in Section~\ref{s:comp} and its performance is evaluated. In Section~\ref{s:exp}, the motivating problem is described, and the paper concludes with a discussion of possible improvements in modelling and methodology.


\begin{figure}[ht]
\centering
\includegraphics[width=\textwidth]{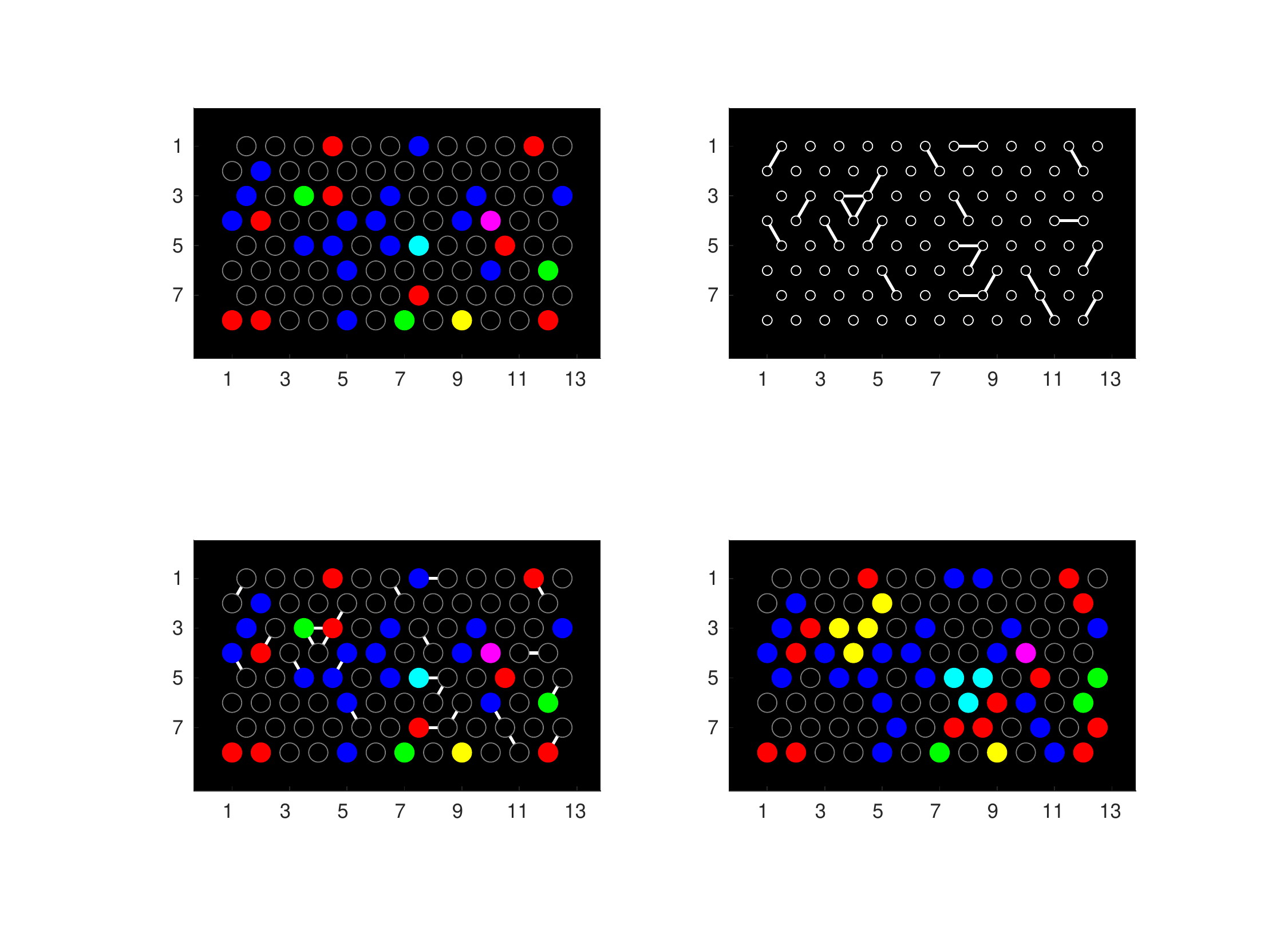}
\caption{(\emph{top left})~A realisation of random seeding $(X^\ell_i)$ with $(\lambda^{\textrm{red}},\lambda^{\textrm{green}},\lambda^{\textrm{blue}})=(0.1,0.05,0.2)$. (\emph{top right})~A realisation of bond percolation on the triangular lattice with $\mu=0.1$. (\emph{bottom left})~The bond percolation overlaid with the seeding. (\emph{bottom right})~The resulting configuration $\big(\widetilde{Y}^\ell_i\big)$ which serves as the data.}\label{f:process}
\end{figure}

\section{Method of simulated moments (MSM)}\label{s:MSM}

The MSM is a modification of the classical method of moments for parameter estimation for the case when the moments of the sampling distribution cannot be computed from the parameters in closed form. The MSM proposes to simulate $n_s$ independent, identically distributed samples from the distribution, repeatedly with different parameter values $\theta$ (usually, but not strictly necessarily, with common random variables as $\theta$ is changed), and to choose the $\theta$ which gives the closest match between moments of the data and that of the simulated data. For its detailed description, we recommend perusing a combination of \cite{Gourieroux_Monfort_1991} and~\cite{Gourieroux_Monfort_1996}.

The data $\mathcal{Y}=(\mathcal{Y}_i)_{i\in I}$ originates from a distribution which is parameterised by the unknown $\theta_0\in\Theta$. $\theta_0$ is called the true value of the parameter. Normally, the $\mathcal{Y}_i$ are independent. A sample from this family of distributions with a general parameter is denoted by $Y=(Y_i)_{i\in I}$. Let $K$ be some $n_m$-dimensional function of the individual observations $Y_i$. Let $k(\theta)$ be the expectation of~$K$ when $K$ is evaluated on a draw $Y_i$ from the distribution with parameter $\theta\in\Theta$, $k(\theta):=\E_\theta[K(Y_i)]$. Thus $k$ is a vector of $n_m$ generalised moments of the distribution of $Y_i$. ($\E_\theta$ is the expectation under the distribution with parameter $\theta$. Similarly, $\P_\theta$ is the probability of an event in that case.)

Let $g$ be some multidimensional function that represents estimating constraints. In our case these are distances between observed moments and moments of the model with given parameter value $\theta$:
\[g(\mathcal{Y}_i,\theta)=K(\mathcal{Y}_i)-k(\theta).\]
By introducing $\E_0$ as a shorthand for $\E_{\theta_0}$, it is immediate that $\E_0[g(\mathcal{Y}_i,\theta_0)]=0$. However, for the parameter estimation problem to be well posed, we require that
\begin{align}\label{e:uniq}
\E_0[g(\mathcal{Y}_i,\theta)]=0\quad\iff\quad\theta=\theta_0.
\end{align}
Implicit in this is that we have at least as many independent equations as parameters.


The MSM is used when $k(\theta)$ is not available in analytical form but there exists an unbiased estimator $\widetilde{k}(U^s_i,\theta)$, and consequently an unbiased estimator for $g$, $\widetilde{g}(\mathcal{Y}_i,U^s_i,\theta)=K(\mathcal{Y}_i)-\widetilde{k}(U^s_i,\theta)$. Here $(U^s_i)_{i\in I,s\in\{1,\dots,n_s\}}$ is some source of randomness, typically vectors of independent, uniform random variables on $[0,1]$ as provided by a pseudorandom number generator. The estimators satisfy $\E\big[\widetilde{k}(U^s_i,\theta)\big]=k(\theta)$ and $\E\big[\widetilde{g}(\mathcal{Y}_i,U^s_i,\theta)\,\big|\,\mathcal{Y}_i\big]=g(\mathcal{Y}_i,\theta)$.


We introduce a weighting by a symmetric, positive definite matrix $\Omega\in\mathbb{R}^{n_m\times n_m}$, which might be a function of the data, and consider the quadratic form $\alpha(\eta)=\eta\T \Omega \eta$. The broad principle of the MSM is the following.


\begin{prop}\label{p:MSM}
The MSM estimator is defined as
\begin{align*}
\hat{\theta}_{n_s,n_I} :&=\argmin{\theta\in\Theta}\, \alpha\left(\frac{1}{n_I}\sum_{i=1}^{n_I}\left(K(\mathcal{Y}_i)-\frac{1}{n_s}\sum_{s=1}^{n_s} \widetilde{k}(U^s_i,\theta)\right)\right).
\end{align*}
If identifiability holds, $n_s$ is fixed and $n_I$ tends to infinity, and the almost sure convergence guaranteed by the SLLN
\begin{align}\label{e:SLLN}
\frac{1}{n_I}\sum_{i=1}^{n_I} \widetilde{k}(U^s_i,\theta) \quad&\conv{n_I\to\infty}\quad k(\theta)
\end{align}
is uniform in $\theta\in\Theta$ for every $s$, then $\hat{\theta}_{n_s,n_I}$ is strongly consistent (that is, $\hat{\theta}_{n_s,n_I}$ converges to $\theta_0$ almost surely).
\end{prop}
Notice that the number of simulations $n_s$ can remain bounded, it is only $n_I$ that must tend to infinity for consistence. For practical implementations, it is a crucial point that the $(U^s_i)$ must be drawn at the beginning of the exploration of the parameter space and kept fixed afterwards while different parameter values are proposed, in order to avoid introducing an extra layer of fluctuation \cite[p.~29]{Gourieroux_Monfort_1996}. This way, a gradient-based search of the parameter space is possible. At the theoretical level, in the limit $n_I\to\infty$, the estimator is strongly consistent even without using common random numbers.

Under the additional condition that $\widetilde{g}(\mathcal{Y}_i,U^s_i,\theta)$ is twice differentiable with respect to $\theta$, asymptotic normality of the estimator also holds and the asymptotic variance can be explicitly given \cite{Gourieroux_Monfort_1991, Gourieroux_Monfort_1996}.\\

For the MSM applied to our percolation model with colouring, the data points $\widetilde{Y}_i$ are neither identically distributed (because of boundary effects) nor independent, and Proposition~\ref{p:MSM} in its current form does not imply the validity of the method. The main theoretical result of this paper is the proof of the strong consistence of a particular MSM estimator for our estimation problem.

The generalised moment function $K$ we propose contains, in addition to first moments $Y^\ell_i$, products $Y^\ell_i Y^\ell_j$ for $i\sim j$ because these carry much information about open edges. We note the consequence that it no longer suffices that $K$ is a function of individual $Y^\ell_i$ only.

We assume without proof that for this generalised moment function, identifiability~\eqref{e:uniq} holds. For supporting evidence, turn to Section~\ref{s:ident} of the Appendix. This assumption is not true in some extreme cases which we exclude. If $(\lambda^1,\dots,\lambda^{n_c})=h\in\{0,1\}^{n_c}$, then $\widetilde{Y}$ is almost surely identically $h$ for any choice of~$\mu$ (and so is~$Y$). For an $h\in\{0,1\}^{n_c}$, the outcome~$\widetilde{Y}$ is again $h$ with high probability as $n_I\to\infty$, if $\mu=1$, and $\lambda^\ell>0$ if and only if $h_\ell =1$.

The percolation parameter $\mu$ is allowed to take any value in the subcritical regime $[0,p_c[$. $p_c$ is the \emph{critical probability} of bond percolation. For the triangular lattice, its value is $p_c=2\sin\frac{\pi}{18}\approx 0.3473$, while for the square lattice, it is $p_c=1/2$ \cite{Sykes_Essam_1964},~\cite[Chapter~3]{Grimmett_1999}.

Section~\ref{s:SLLN} details the steps leading to the SLLN result~\eqref{e:SLLN}. Due to dependence between the~$Y_i$, cross-correlations appear in the derivation in addition to variances. Section~\ref{s:corr} deals with upper bounding these correlations using percolation theory. Section~\ref{s:ULLN} describes the extension of the SLLN to ULLN.

The observed colouring of the dataset is denoted by $\mathcal{Y}^\ell_i$ ($i\in I$, $\ell\in\{1,2,\dots,n_c\}$), whereas in the simulated data it is $\widetilde{Y}^{\ell,s}_i$ ($s\in\{1,2,\dots,n_s\}$). While it is clear that the simulated data must come from a finite~$I$ (or perhaps from some $I':\ I\subset I'\subset L$), we leave flexibility whether the data is of type $\big(\widetilde{\mathcal{Y}_i}\big)_{i\in I}$, which is the case in our practical application, or of the theoretically appealing type~$(\mathcal{Y}_i)_{i\in I}$. We let $(\mathcal{Y}_i)_{i\in I}$ denote both cases, to be interpreted as the context demands. Lastly, we introduce the following averages: 
\begin{align*}
\bar{\mathcal{Y}}^\ell&:=\frac{1}{n_I}\sum_{i\in I}\mathcal{Y}^\ell_i,&\bar{Y}^{\ell,s}&:=\frac{1}{n_I}\sum_{i\in I}\widetilde{Y}^{\ell,s}_i,\\
\bar{\mathcal{Z}}^\ell&:=\frac{1}{n_p}\sum_{(i,j)\in I_2}\mathcal{Y}^\ell_i \mathcal{Y}^\ell_j,&\bar{Z}^{\ell,s}&:=\frac{1}{n_p}\sum_{(i,j)\in I_2}\widetilde{Y}^{\ell,s}_i \widetilde{Y}^{\ell,s}_j.
\end{align*}
Our main theorem is the following.

\begin{thm}\label{th:main}
Let $\Theta$ be a compact subset of $([0,1]^{n_c}\setminus \{0,1\}^{n_c})\times [0,p_c[$. (For the triangular lattice, $p_c=2\sin\frac{\pi}{18}\approx 0.3473$, while in the square lattice case, $p_c=1/2$.) Consider the bond percolation model with colouring and with the true parameter value $\theta_0=(\lambda^1,\dots,\lambda^{n_c},\mu)\in\Theta$. Let $\Omega\in\mathbb{R}^{2n_c\times 2n_c}$ be a symmetric, positive definite matrix, which might be a function of the data, and write $\alpha(\eta)=\eta\T \Omega \eta$ for the resulting quadratic form. Under the assumption of identifiability, when $n_s$ is fixed and $n_I$ tends to infinity,
\begin{align*}
\hat{\theta}_{n_s,n_I} :&=\argmin{\theta\in\Theta}\,\alpha\left(\begin{array}{c} \left(\frac{1}{n_I}\sum_{i\in I}\left(\mathcal{Y}^\ell_i-\frac{1}{n_s}\sum_{s=1}^{n_s}\widetilde{Y}^{\ell,s}_i \right)\right)_{\ell\in\{1,\dots,n_c\}}\\ \left(\frac{1}{n_p}\sum_{(i,j)\in I_2}\left(\mathcal{Y}^\ell_i \mathcal{Y}^\ell_j-\frac{1}{n_s}\sum_{s=1}^{n_s}\widetilde{Y}^{\ell,s}_i \widetilde{Y}^{\ell,s}_j \right)\right)_{\ell\in\{1,\dots,n_c\}}\end{array} \right)\\
&=\argmin{\theta\in\Theta}\,\alpha\left(\begin{array}{c} \left(\bar{\mathcal{Y}}^\ell - \frac{1}{n_s}\sum_{s=1}^{n_s} \bar{Y}^{\ell,s} \right)_{\ell\in\{1,\dots,n_c\}}\\ \left(\bar{\mathcal{Z}}^\ell - \frac{1}{n_s}\sum_{s=1}^{n_s} \bar{Z}^{\ell,s} \right)_{\ell\in\{1,\dots,n_c\}}\end{array} \right)
\end{align*}
is strongly consistent.
\end{thm}

In order to prove the claim, we want to establish that for the arithmetic means generated under general~$\theta$, the following almost sure convergences hold as $n_I\to\infty$, uniformly in~$\theta\in\Theta$:
\begin{align*}
\frac{1}{n_I}\sum_{i\in I}Y^\ell_i - \frac{1}{n_I}\sum_{i\in I}\E_\theta Y^\ell_i &\ \longrightarrow\ 0\\
\mathrm{and}\quad \frac{1}{n_p}\sum_{(i,j)\in I_2}Y^\ell_i Y^\ell_j - \frac{1}{n_p}\sum_{(i,j)\in I_2} \E_\theta\!\left[Y^\ell_i Y^\ell_j\right]&\ \longrightarrow\ 0,
\end{align*}
for $i\sim j$. The same proofs apply with~$\widetilde{Y}$, too. This unusual formulation of the SLLN is needed because the random variables~$\widetilde{Y}$ are not identically distributed due to boundary effects. These two SLLNs ultimately ensure that
\begin{multline}\label{e:convalpha}
\alpha\left(\begin{array}{c} \left(\bar{\mathcal{Y}}^\ell - \frac{1}{n_s}\sum_{s=1}^{n_s} \bar{Y}^{\ell,s} \right)_{\ell\in\{1,\dots,n_c\}}\\ \left(\bar{\mathcal{Z}}^\ell - \frac{1}{n_s}\sum_{s=1}^{n_s} \bar{Z}^{\ell,s} \right)_{\ell\in\{1,\dots,n_c\}}\end{array} \right)-\\
\alpha\left(\begin{array}{c} \left(\frac{1}{n_I}\sum_{i\in I}\left(\E_0 Y^\ell_i - \E_\theta \widetilde{Y}^\ell_i \right)\right)_{\ell\in\{1,\dots,n_c\}}\\ \left(\frac{1}{n_p}\sum_{(i,j)\in I_2}\left(\E_0[Y^\ell_i Y^\ell_j] - \E_\theta [\widetilde{Y}^\ell_i \widetilde{Y}^\ell_j] \right)\right)_{\ell\in\{1,\dots,n_c\}}\end{array} \right)\ \conv{n_I\to\infty}\ 0
\end{multline}
uniformly with probability~$1$. The right term is minimal when it is asymptotically zero (in the case when $\E_0$ acts on $Y^\ell_i$ and $Y^\ell_i Y^\ell_j$; when it acts on $\widetilde{Y}^\ell_i$ and $\widetilde{Y}^\ell_i \widetilde{Y}^\ell_j$, then it is actually zero), and this is achieved in only $\theta=\theta_0$ under the assumption of identifiability~\eqref{e:uniq}. This gives the strong consistence for~$\hat{\theta}_{n_s,n_I}$.


\section{Strong law of large numbers with weak dependence}\label{s:SLLN}

We adapt the proof of Theorem~1 of~\cite{Etemadi_1983a} in this section to suit our purposes. We write out the claims with~$Y$, but they also hold for~$\widetilde{Y}$.
\begin{prop}\label{p:first}
Let $\theta\in [0,1]^{n_c}\times [0,p_c[$, where $p_c$ is the critical probability of bond percolation. If $Y$ is generated with parameter value $\theta$, then
\[\frac{1}{n_I}\left(\sum_{i\in I}Y^\ell_i-\sum_{i\in I}\E_\theta Y^\ell_i\right)\ \conv{n_I\to\infty}\ 0\]
almost surely. The claim also holds for~$\widetilde{Y}$.
\end{prop}
\begin{prop}\label{p:second}
Let $\theta\in [0,1]^{n_c}\times [0,p_c[$. If $Y$ is generated with parameter value $\theta$, then
\[\frac{1}{n_p}\left(\sum_{(i,j)\in I_2}Y^\ell_i Y^\ell_j-\sum_{(i,j)\in I_2}\E_\theta [Y^\ell_i Y^\ell_j]\right)\ \conv{n_I\to\infty}\ 0\]
almost surely. The claim also holds for~$\widetilde{Y}$.
\end{prop}

\begin{proof}[Proposition~\ref{p:first}] For the ease of notation, let $Y_i:=Y^\ell_i$ for some fixed $\ell\in\{1,\dots,n_c\}$ ($i\in I$), created by our percolation process with $\theta=(\lambda^1,\dots,\lambda^{n_c},\mu)$. Let $a>1$ and define the lacunary sequence $k_n:=[a^n]$. Let $S_k:=\sum_{i=1}^k Y_i$.

By the application of Chebyshov's inequality, for every $\varepsilon>0$,
\begin{align}
\sum_{n=1}^\infty\P\left(\left|\frac{S_{k_n}-\E S_{k_n}}{k_n}\right|>\varepsilon\right)&\le\sum_{n=1}^\infty\frac{\mathrm{Var}\, S_{k_n}}{\varepsilon^2 k_n^2}\notag\\
&\le\frac{1}{\varepsilon^2}\sum_{n=1}^\infty \frac{1}{k_n^2}\sum_{i=1}^{k_n}\mathrm{Var}\, Y_i\notag\\
&\quad+\frac{1}{\varepsilon^2}\sum_{n=1}^\infty\frac{1}{k_n^2}\sum_{1\le i\neq j\le k_n}(\E[Y_i Y_j]-\E Y_i \,\E Y_j).\label{e:Cheb}
\end{align}
If we can prove that this is finite, then by the Borel--Cantelli lemma, as $n\to\infty$, for every $\theta\in\Theta$,
\begin{align}\label{e:knconv}
\left|\frac{S_{k_n}-\E S_{k_n}}{k_n}\right|&\to 0\quad\textrm{a.s.}
\end{align}
We first show that
\[\sum_{n=1}^\infty \frac{1}{k_n^2}\sum_{i=1}^{k_n}\mathrm{Var}\, Y_i<\infty\]
by noticing that $\sup_{i\in I}\mathrm{Var}\, Y_i\le 1$ and by the following lemma. 
\begin{lem}\label{l:sumkn}
If $1<a$, then
\[\sum_{n=1}^\infty \frac{1}{k_n}<\infty.\]
\end{lem}
\begin{proof}
For $n\in\mathbb{N}$ sufficiently large, $a^n/2\le k_n$ because $n\ge \frac{\log 2}{\log a}$ suffices. To see this, consider that $a^n/2 \le a^n -1 <k_n$ is achieved, giving the threshold, if $2\le a^n$. Let
\[N_0:=\max\left\{1,\left\lceil \frac{\log 2}{\log a}\right\rceil\right\}.\]
Consequently, for some constant $c$,
\begin{align*}
\sum_{n=1}^\infty \frac{1}{k_n}&=\sum_{n=1}^{N_0 -1} \frac{1}{k_n} + \sum_{n=N_0}^\infty \frac{1}{k_n} \le c + \sum_{n=N_0}^\infty \frac{2}{a^n} = c+ \frac{2}{a^{N_0}(1-1/a)}<\infty.
\end{align*}
\end{proof}


We prove in Section~\ref{s:corr} that
\begin{align}
\left|\sum_{1\le i\neq j\le k_n}(\E[ Y_i Y_j]-\E Y_i \,\E Y_j)\right|&=\mathcal{O}(k_n),\label{e:corr}
\end{align}
so that by applying Lemma~\ref{l:sumkn} once again, we get that \eqref{e:Cheb} is finite, as required.

In the case of a general $k:=n_I$, $k$ is sandwiched between some $k_n\le k < k_{n+1}$ and
\begin{align}
\frac{S_k-\E S_k}{k}&\le \frac{S_{k_{n+1}}-\E S_{k_n}}{k}\notag\\
&\le\left| \frac{S_{k_{n+1}}-\E S_{k_{n+1}}}{k_{n+1}}\right|\frac{k_{n+1}}{k_n} + \frac{\E S_{k_{n+1}}-\E S_{k_n}}{k_n}.\label{e:genk}
\end{align}
Note that even for $S_{k_{n+1}}-\E S_{k_n}<0$, one can change the denominator from $k$ to~$k_n$ in the second inequality because the right-hand side is nonnegative. Here, for a fixed $a>1$,
\begin{align}
\frac{k_{n+1}}{k_n}&=\frac{[a^{n+1}]}{[a^n]}\le\frac{a^{n+1}}{a^n -1}=a+\frac{a}{a^n -1},\label{e:succlacun}
\end{align}
which in turn is arbitrarily close to $a$ when $n$ is sufficiently large. Additionally,
\begin{align*}
\frac{\E S_{k_{n+1}}-\E S_{k_n}}{k_n}&\le\frac{(k_{n+1}-k_n)\sup_{i\in I}\E Y_i}{k_n}\\
&\le \left(a+\frac{a}{a^n -1}-1\right)\sup_{i\in I}\E Y_i,
\end{align*}
and combining this with \eqref{e:knconv} yields
\begin{align*}
\mathop{\mathrm{lim\,sup}}_{k\to\infty}\frac{S_k-\E S_k}{k}\le (a-1)\sup_{i\in I}\E Y_i\le a-1.
\end{align*}
A similar lower bound can also be attained. Since $a>1$ can be chosen arbitrarily, the SLLN for $Y_i$ (Proposition~\ref{p:first}) holds once we prove the estimate~\eqref{e:corr}.
\end{proof}

\begin{proof}[Proposition~\ref{p:second}] 
This proof goes entirely analogously to that of Proposition~\ref{p:first}. We keep using the notation $Y_i:=Y^\ell_i$ for some fixed $\ell\in\{1,\dots,n_c\}$ and fixed~$\theta$, and the lacunary sequence $k_n=[a^n]$ for $a>1$. Let $T_k:=\sum_{(i,j)\in I_2} Y_i Y_j$ for $I=I(k)$ composed of the first $k$ vertices according to the fixed ordering. This sum has $n_p(k)$ terms. Then, by the argument of \eqref{e:Cheb}, for every $\varepsilon>0$,
\begin{align}
\sum_{n=1}^\infty\P\left(\left|\frac{T_{k_n}-\E T_{k_n}}{n_p(k_n)}\right|>\varepsilon\right)&
\le\sum_{n=1}^\infty\frac{\mathrm{Var}\, T_{k_n}}{\varepsilon^2 n_p(k_n)^2}\notag\\
&\le\frac{1}{\varepsilon^2}\sum_{n=1}^\infty \frac{1}{n_p(k_n)^2}\sum_{(i_1,i_2)\in I_2(k_n)}\mathrm{Var}[Y_{i_1}Y_{i_2}]&&\notag\\
&\quad+\frac{1}{\varepsilon^2}\sum_{n=1}^\infty\frac{1}{n_p(k_n)^2}\sum_{\substack{(i_1,i_2),(j_1,j_2)\in I_2(k_n)\\(i_1,i_2)\neq(j_1,j_2)}}\Big(\E[Y_{i_1}Y_{i_2} Y_{j_1}Y_{j_2}]\notag\\
&\qquad-\E [Y_{i_1}Y_{i_2}] \,\E [Y_{j_1}Y_{j_2}]\Big).\label{e:Cheb2}
\end{align}
By
\[\sup_{(i_1,i_2)\in I_2(k_n)}\mathrm{Var}[Y_{i_1} Y_{i_2}]\le 1\]
and $|I_2(k_n)|=n_p(k_n)\sim 3 n_I = 3 k_n$, Lemma~\ref{l:sumkn} gives
\[\frac{1}{\varepsilon^2}\sum_{n=1}^\infty \frac{1}{n_p(k_n)^2}\sum_{(i_1,i_2)\in I_2(k_n)}\mathrm{Var}[Y_{i_1}Y_{i_2}]\le \frac{1}{\varepsilon^2}\sum_{n=1}^\infty \frac{1}{n_p(k_n)}<\infty.\]
In Section~\ref{s:corr}, it is shown that
\begin{align}
\left|\sum_{\substack{(i_1,i_2),(j_1,j_2)\in I_2(k_n)\\(i_1,i_2)\neq (j_1,j_2)}}\Big(\E[Y_{i_1}Y_{i_2} Y_{j_1}Y_{j_2}]-\E [Y_{i_1}Y_{i_2}] \,\E [Y_{j_1}Y_{j_2}]\Big)\right|&=\mathcal{O}(k_n),\label{e:corr2}
\end{align}
and by Lemma~\ref{l:sumkn}, we get that the sum \eqref{e:Cheb2} is finite. By the Borel--Cantelli lemma,
\begin{align}\label{e:knconv2}
\left|\frac{T_{k_n}-\E T_{k_n}}{n_p(k_n)}\right|&\to 0\quad\textrm{a.s.}
\end{align}

For a general $k=n_I$ with $k_n\le k<k_{n+1}$, 
\begin{align}
\frac{T_k-\E T_k}{n_p(k)}\le \left| \frac{T_{k_{n+1}}-\E T_{k_{n+1}}}{n_p(k_{n+1})}\right|\frac{n_p(k_{n+1})}{n_p(k_n)} + \frac{\E T_{k_{n+1}}-\E T_{k_n}}{n_p(k_n)}.\label{e:split}
\end{align}
For a fixed $a>1$, by using \eqref{e:succlacun} again,
\begin{align*}
\frac{n_p(k_{n+1})}{n_p(k_n)}\sim\frac{3k_{n+1}}{3k_n}\le a+\frac{a}{a^n -1},
\end{align*}
and the right-hand side is arbitrarily close to $a$ when $n$ is sufficiently large. Additionally,
\begin{align*}
\frac{\E T_{k_{n+1}}-\E T_{k_n}}{n_p(k_n)}&\le\frac{(n_p(k_{n+1})-n_p(k_n))\sup_{(i_1,i_2)\in I_2(k_n)}\E [Y_{i_1}Y_{i_2}]}{n_p(k_n)},
\end{align*}
hence
\begin{align*}
\mathop{\mathrm{lim\,sup}}_{n\to\infty}\frac{\E T_{k_{n+1}}-\E T_{k_n}}{n_p(k_n)}&\le(a-1)\sup_{(i_1,i_2)\in I_2(k_n)}\E[Y_{i_1}Y_{i_2}].
\end{align*}
Combining this with \eqref{e:knconv2} and \eqref{e:split}, we get
\begin{align*}
\mathop{\mathrm{lim\,sup}}_{k\to\infty}\frac{T_k-\E T_k}{n_p(k)}\le (a-1)\sup_{(i_1,i_2)\in I_2(k_n)}\E[Y_{i_1}Y_{i_2}]\le a-1.
\end{align*}
A similar lower bound can also be attained. Since $a>1$ can be chosen arbitrarily, the SLLN for $Y_i Y_j$, $i\sim j$ (Proposition~\ref{p:second}) holds once we prove the estimate~\eqref{e:corr2}.
\end{proof}

\section{Upper bound on correlations}\label{s:corr}

We prove the estimates \eqref{e:corr} and~\eqref{e:corr2} in greater generality, for every positive integer~$n$. Let $\Theta$ be a compact subset of $[0,1]\times [0,p_c[$, where $p_c$ is the critical probability of bond percolation.
\begin{lem}\label{l:corr}
As $n\to\infty$, it holds
\begin{align*}
\sup_{\theta\in\Theta}\left|\sum_{1\le i\neq j\le n}(\E[Y_i Y_j]-\E Y_i \,\E Y_j)\right|&=\mathcal{O}(n).
\end{align*}
\end{lem}

\begin{lem}\label{l:corr2}
As $n\to\infty$, it holds
\begin{align*}
\sup_{\theta\in\Theta}\left|\sum_{\substack{(i_1,i_2),(j_1,j_2)\in I_2(n)\\(i_1,i_2)\neq (j_1,j_2)}}\Big(\E[Y_{i_1}Y_{i_2} Y_{j_1}Y_{j_2}]-\E [Y_{i_1}Y_{i_2}] \,\E [Y_{j_1}Y_{j_2}]\Big)\right|&=\mathcal{O}(n).
\end{align*}
\end{lem}

For background, first we recapitulate from the fundamentals of percolation theory the meaning of increasing events, the FKG inequality, disjoint occurrence, the BK inequality and pivotality \cite[Chapter~2]{Grimmett_1999}. It is well known that these concepts do not rely on the specific structure of the lattice graph and can be cast more generally in terms of functions of Boolean variables. 

In this vein, one can consider a probability space $(\Gamma,\mathcal{F},\P)$ with sample space $\Gamma =\{0,1\}^S$ ($S$ is finite or at most countably infinite) where the set of events $\mathcal{F}$ is the $\sigma$-algebra generated by the finite-dimensional cylinder sets and the measure is a product measure
\[\P=\prod_{s\in S} \nu_s\]
where $\nu_s$ is specified by some vector $(p(s))_{s\in S}\in [0,1]^S$ via
\[\nu_s(\omega(s)=1)=p(s),\quad\nu_s(\omega(s)=0)=1-p(s)\]
for sample vectors $(\omega(s))_{s\in S}\in \{0,1\}^S$ \cite[Chapter~2, p.~33]{Grimmett_1999}.

In our application, we have already fixed a colour $\ell\in\{1,\dots,n_c\}$ and look at colours independently. We extend the set of vertices $L$ with an additional vertex that we call $\infty^\ell$, or simply $\infty$ when the colour is fixed and unimportant: $L^*:=L\cup\{\infty\}$. We also extend the edge set of the triangular lattice $L_2$ with edges between each vertex and $\infty$, and the value assigned to such an edge indicates the presence or absence of seeding. We call these edges \emph{source edges}. For the source edges, $p(s)=\lambda^\ell$, and for the edges of the lattice which represent contamination, $p(s)=\mu$. The interpretation is that $Y^\ell_i=1$ if and only if $i\leftrightarrow^* \infty^\ell$, where the asterisk refers to connection in the extended graph.

An event $A\in\mathcal{F}$ of the $\sigma$-algebra is called \emph{increasing}, if whenever $\omega \le \omega'$, $\omega\in A$ implies $\omega'\in A$.

\begin{thm}[FKG inequality \cite{Fortuin_Kasteleyn_Ginibre_1971},{\cite[pp. 34--36]{Grimmett_1999}}] If $A$ and $B$ are increasing events in $\mathcal{F}$, then $\P(A\cap B)\ge\P(A)\P(B)$.
\end{thm}

Let $e_1,e_2,\dots,e_N$ be $N$ distinct edges of the graph, and $A,B\in\mathcal{F}$ two increasing events which depend on the vector of the states of these $N$ edges $\omega=(\omega(e_1),\dots,\omega(e_N))$ only. Such vectors $\omega$ are characterised uniquely by the set of edges with value~$1$: $J(\omega)=\big\{e_i\,\big|\,i\in \{1,\dots,N\},\, \omega(e_i)=1\big\}$.

For the increasing events $A,B$, the event $A\circ B$ (we say \emph{$A$ and $B$ occur disjointly}) is the set of all $\omega\in\Gamma$ for which there exists an $H\subseteq J(\omega)$ such that $\omega'$ determined by $J(\omega')=H$ belongs to $A$, and $\omega''$ determined by $J(\omega'')=J(\omega)\setminus H$ belongs to $B$. In words, $A\circ B$ is the set of assignments of $0$ and $1$ to the edges for which there exist two disjoint sets of edges assigned the value~$1$ (\emph{open} edges) such that the first such set ensures the occurrence of event~$A$ and the second set ensures the occurrence of~$B$. It is easy to verify that $A\circ B$ is also increasing and $A\circ B\subseteq A\cap B$.

The classical example for disjoint occurrence is when $A$ is the event that there is an open path joining $i_1$ to~$j_1$ within the finite subgraph given by $\{e_1,\dots,e_N\}$ and $B$ is the event that there is an open path between $i_2$ and~$j_2$ within the same finite subgraph. Then $A\circ B$ is the event that there exist two edge-disjoint paths, the first between $i_1$ and~$j_1$ and another one joining $i_2$ to~$j_2$.

\begin{thm}[BK inequality \cite{vandenBerg_Kesten_1985},{\cite[pp. 37--41]{Grimmett_1999}}] If $A$ and $B$ are increasing events in $\mathcal{F}$, then $\P(A\circ B)\le\P(A)\P(B)$.
\end{thm}

The validity of the inequality extends to the existence of arbitrarily long (but finite length) edge-disjoint open paths, which is what we need it for, by taking a sequence of growing, nested subsets of~$L$ \cite[p.~38]{Grimmett_1999}.

The notion of pivotality is not used until Section~\ref{s:ULLN}. For any event~$A$ an edge~$e$ is pivotal if its open or closed state is crucial to whether $A$ occurs or not. In more detail, the edge $e$ is pivotal for the pair $(A,\omega)$, if for the indicator function of~$A$, $\chi_A(\omega)\neq\chi_A(\omega')$, where the configuration $\omega'\in \{0,1\}^S$ is defined by $\omega'(e)=1-\omega(e)$, and $\omega'(f)=\omega(f)$ for every edge $f\neq e$. The event that $e$ \emph{is pivotal for} $A$ is the set of $\omega$ for which $e$ is pivotal for $(A,\omega)$.

\begin{proof}[Lemma~\ref{l:corr}] In the extended lattice graph that has source edges with weight zero or one at every vertex for seeding, the event $\{Y_i=1\}$ for $i\in I$ is increasing because it is increasing in both seeding (source edges) and contamination edges. For any $i,j\in I$,
\begin{align*}
\E[Y_i Y_j]-\E Y_i \,\E Y_j&=\P(Y_i Y_j=1)-\P(Y_i=1) \,\P(Y_j=1)\ge 0
\end{align*}
by the FKG inequality. Hence, for every $\theta\in\Theta$,
\begin{align*}
\sum_{1\le i\neq j\le n}(\E[Y_i Y_j]-\E Y_i \,\E Y_j)&\ge 0.
\end{align*}

For the upper bound, consider that
\begin{align}
\P(Y_i Y_j=1)-\P(Y_i=1) \,\P(Y_j=1)&=\P\big(\{Y_i=1\}\circ\{Y_j=1\}\big)-\P(Y_i=1) \,\P(Y_j=1)\notag\\
&\quad +\P\Big(\{Y_i Y_j=1\}\setminus \{Y_i=1\}\circ\{Y_j=1\}\Big)\notag\\
&\le \P\Big(\{Y_i Y_j=1\}\setminus \{Y_i=1\}\circ\{Y_j=1\}\Big)\label{e:BK}
\end{align} 
by the BK inequality. Cooccurrence of $\{Y_i=1\}$ and $\{Y_j=1\}$ which is not disjoint is one where $i$ and $j$ are in the same component in the edge set on the non-extended lattice:
\[\{Y_i Y_j=1\}\setminus \{Y_i=1\}\circ\{Y_j=1\} \subseteq \{i\leftrightarrow j\}.\] 
We show that
\begin{equation}
\sum_{1\le i\neq j\le n}\P(i\leftrightarrow j)=\mathcal{O}(n)\label{e:sumofpaths}
\end{equation}
for $\mu<p_c$, and uniformly so for $\mu\in[0,p_c-\varepsilon]$ for every $\varepsilon>0$. This follows from the exponential decay of the cluster size distribution and it will complete the proof of Lemma~\ref{l:corr}.

Let $C(i)$ denote the set of vertices in the component of~$i\in L$ according to the non-extended edge set of~$L$. Then
\begin{align*}
\sum_{1\le i\neq j\le n}\P(i\leftrightarrow j)&=\sum_{1\le i\le n}\,\sum_{\substack{1\le j\le n\\j\neq i}}\E\chi_{\{i\leftrightarrow j\}}=\sum_{1\le i\le n}\E[|C(i)|-1].
\end{align*}
\begin{thm}[Exponential decay of the cluster size distribution \cite{Aizenman_Newman_1984}, {\cite[Chapter~6]{Grimmett_1999}}] For $\mu\in]0,p_c[$, there exists $g(\mu)>0$ such that for all $k\ge 1$ and $i\in I$, for the bond percolation with parameter~$\mu$, it holds that $\P(|C(i)|\ge k)\le \mathrm{e}^{-kg(\mu)}$. 
\end{thm}
Take $\mu^*=p_c-\varepsilon$. As $\P(|C(i)|\ge k)$ is nondecreasing in~$\mu$, we get a uniform bound in~$\theta\in\Theta$ if the bound is valid for~$\mu^*$:
\begin{align}
\sum_{i=1}^n \E[|C(i)|-1]&=\sum_{i=1}^n\left(\left(\sum_{k=1}^\infty P(|C(i)|\ge k)\right)-1\right)\notag\\
&\le \sum_{i=1}^n \sum_{k=1}^\infty \mathrm{e}^{-kg(\mu^*)} = n\frac{1}{\mathrm{e}^{g(\mu^*)}-1}.\label{e:samecomp} 
\end{align}
This proves Lemma~\ref{l:corr}, which in turn completes the proof of Proposition~\ref{p:first}.
\end{proof}



To go from the case of $Y$ to~$\widetilde{Y}$, first we couple the realisations of $(Y_i)_{i\in L}$ and~$\big(\widetilde{Y}_i\big)_{i\in I}$ with varying lattices~$I$ (and later with varying parameter vectors) by defining them via shared random variables $(U^\ell_i)_{i\in L,\ell\in\{1,2,\dots,n_c\}}$ and $(V_{ij})_{(i,j)\in L_2}$ that are independent and all uniformly distributed on $[0,1]$. For $\theta=(\lambda^1,\dots,\lambda^{n_c},\mu)\in\Theta$, $i\in L$, $(i,j)\in L_2$ and $\ell\in\{1,2,\dots,n_c\}$, the seeding is defined by $X^\ell_i:=\chi_{\{U^\ell_i<\lambda^\ell\}}$, and edges are open according to $\xi_{ij}:=\chi_{\{V_{ij}<\mu\}}$.

Let us drop the superscript $\ell$ again. Notice that any $\widetilde{Y}_i$ can increase when $I$ is increased. In the proof of Proposition~\ref{p:first}, the only occasion where $Y_i$ from different~$I$ are compared is inequality~\eqref{e:genk}. We mark the lattice size as a variable in the superscript of~$\widetilde{Y}_i$. Observe that $\widetilde{Y}^{k_n}_i\le \widetilde{Y}^k_i\le \widetilde{Y}^{k_{n+1}}_i\le Y_i$ for $k_n\le k < k_{n+1}$ and $i\in I(k_n)$. With $\widetilde{S}^k_n:=\sum_{i=1}^n \widetilde{Y}^k_i$, noting $\widetilde{S}^{k_n}_{k_n}\le \widetilde{S}^k_k\le \widetilde{S}^{k_{n+1}}_{k_{n+1}}$,
\begin{align*}
\frac{\widetilde{S}^k_k-\E \widetilde{S}^k_k}{k}&\le \frac{\widetilde{S}^{k_{n+1}}_{k_{n+1}}-\E \widetilde{S}^{k_n}_{k_n}}{k}\\
&\le\left| \frac{\widetilde{S}^{k_{n+1}}_{k_{n+1}}-\E \widetilde{S}^{k_{n+1}}_{k_{n+1}}}{k_{n+1}}\right|\frac{k_{n+1}}{k_n} + \frac{\E \widetilde{S}^{k_{n+1}}_{k_{n+1}}-\E \widetilde{S}^{k_n}_{k_n}}{k_n},
\end{align*}
where, similarly to inequality~\eqref{e:genk}, the second inequality holds for different reasons when $\widetilde{S}^{k_{n+1}}_{k_{n+1}}-\E \widetilde{S}^{k_n}_{k_n}$ is negative and when not. The first term does not require special treatment. The second term is
\begin{align}\label{e:sdl}
\frac{\E \widetilde{S}^{k_{n+1}}_{k_{n+1}}-\E \widetilde{S}^{k_n}_{k_n}}{k_n}&=\frac{1}{k_n}\sum_{i=k_n+1}^{k_{n+1}} \E \widetilde{Y}^{k_{n+1}}_i+ \frac{1}{k_n}\sum_{i=1}^{k_n}\left(\E \widetilde{Y}^{k_{n+1}}_i-\E \widetilde{Y}^{k_n}_i\right).
\end{align}
Here
\begin{align*}
\frac{1}{k_n}\sum_{i=k_n+1}^{k_{n+1}} \E \widetilde{Y}^{k_{n+1}}_i&\le \frac{k_{n+1}-k_n}{k_n}\sup_{i\in I(k_{n+1})} \E \widetilde{Y}^{k_{n+1}}_i
\end{align*}
is dealt with as in the original proof of Proposition~\ref{p:first}. For the other term of~\eqref{e:sdl},
\begin{align*}
\frac{1}{k_n}\sum_{i=1}^{k_n}\left(\E \widetilde{Y}^{k_{n+1}}_i-\E \widetilde{Y}^{k_n}_i\right)&\le \frac{1}{k_n}\sum_{i=1}^{k_n}\left(\E Y_i-\E \widetilde{Y}^{k_n}_i\right).
\end{align*}
According to the next proposition, this vanishes in the limit, leaving us with
\[\mathop{\mathrm{lim\,sup}}_{k\to\infty}\frac{\widetilde{S}^k_k-\E \widetilde{S}^k_k}{k} \le a-1,\]
as required.

\begin{prop}\label{p:means}
For a compact subset $\Theta\subset [0,1]\times [0,p_c[$,
\[\sup_{\theta\in\Theta}\frac{1}{n_I}\left|\sum_{i\in I}\E_\theta \widetilde{Y}_i-\sum_{i\in I}\E_\theta Y_i\right|\ \conv{n_I\to\infty}\ 0.\]
\end{prop}
\begin{proof}
As $Y_i\ge \widetilde{Y}_i$ almost surely,
\begin{align*}
\E\left[Y_i- \widetilde{Y}_i\right]&=\P\left(Y_i=1,\ \widetilde{Y}_i=0\right)\\
&\le\P\left(Y_i=1\textrm{ and }\exists j\in L\setminus I:\ X_j=1,\ i\leftrightarrow j\right)\\
&\le\P\left(i\leftrightarrow \Delta I\right),
\end{align*}
which expresses that $Y_i$ and $\widetilde{Y}_i$ can differ only if $i\in I$ is connected to the exterior vertex boundary of~$I$. Further,
\begin{align*}
\sum_{i\in I}\P\left(i\leftrightarrow \Delta I\right)&=\E\left[\sum_{i\in I}\chi_{\{i\leftrightarrow \Delta I\}}\right].
\end{align*}
But this is the expected size of the open component that is grown from all vertices of $\Delta I$ towards the inside of~$I$. It is upper bounded by $|\Delta I|\times\E[|C(0)|]$. On the compact~$\Theta$, the mean size of the open component of any vertex has a universal finite upper bound by~\eqref{e:samecomp}. Hence,
\begin{align*}
\sup_{\theta\in\Theta}\frac{1}{n_I}\left|\sum_{i\in I}\E_\theta \left[\widetilde{Y}_i-Y_i\right]\right|\le\frac{|\Delta I|}{n_I}\,\E[|C(0)|]\to 0
\end{align*}
as $n_I\to \infty$, due to our assumption $|\Delta I|/|I|\to 0$ about the nested sequence of~$I$.
\end{proof}

As in the case of~$Y$, a lower bound for $\big(\widetilde{S}^k_k-\E \widetilde{S}^k_k\big)/k$ does not pose any additional difficulty. In the proof of Lemma~\ref{l:corr}, the covariances cannot increase when we constrain the set of edges to those among the first $n$ vertices. Concretely, $\E[|C(i)|]$ cannot increase. Therefore Proposition~\ref{p:first} stays true for~$\widetilde{Y}$.

\begin{proof}[Lemma~\ref{l:corr2}] The proof follows closely that of Lemma~\ref{l:corr}. For any two pairs $(i_1,i_2),(j_1,j_2)\in I_2$,
\[\E[Y_{i_1}Y_{i_2} Y_{j_1}Y_{j_2}]-\E [Y_{i_1}Y_{i_2}] \,\E [Y_{j_1}Y_{j_2}]\ge 0\]
due to the FKG inequality applied to $\{Y_{i_1}Y_{i_2}=1\}$ and $\{Y_{j_1}Y_{j_2}=1\}$. Therefore, for any $\theta\in\Theta$,
\begin{align*}
\sum_{\substack{(i_1,i_2),(j_1,j_2)\in I_2(n)\\(i_1,i_2)\neq (j_1,j_2)}}\Big(\E[Y_{i_1}Y_{i_2} Y_{j_1}Y_{j_2}]-\E [Y_{i_1}Y_{i_2}] \,\E [Y_{j_1}Y_{j_2}]\Big)&\ge 0.
\end{align*}

The first step towards the upper bound, similarly to~\eqref{e:BK}, uses the BK inequality:
\begin{multline*}
\P(Y_{i_1}Y_{i_2} Y_{j_1}Y_{j_2}=1)-\P(Y_{i_1}Y_{i_2}=1) \,\P(Y_{j_1}Y_{j_2}=1)\\
\le \P\Big(\{Y_{i_1}Y_{i_2} Y_{j_1}Y_{j_2}=1\}\setminus \{Y_{i_1}Y_{i_2}=1\}\circ\{Y_{j_1}Y_{j_2}=1\}\Big).
\end{multline*}

Cooccurrence which is not disjoint is one where at least one of $i_1$ and $i_2$ is connected to at least one of $j_1$ and $j_2$ in the non-extended edge set, or in symbols,
\begin{multline*}
\{Y_{i_1}Y_{i_2} Y_{j_1}Y_{j_2}=1\}\setminus \{Y_{i_1}Y_{i_2}=1\}\circ\{Y_{j_1}Y_{j_2}=1\}
\subseteq \{i_1\leftrightarrow \{j_1, j_2\}\} \cup \{i_2\leftrightarrow \{j_1, j_2\}\},
\end{multline*}
where $\leftrightarrow$ is meant to be reflective so that not disjoint cooccurrence might involve e.g. that $i_1=j_1$. So for every fixed $\theta\in\Theta$,
\begin{multline*}
\sum_{\substack{(i_1,i_2),(j_1,j_2)\in I_2(n)\\(i_1,i_2)\neq(j_1,j_2)}}\Big(\E[Y_{i_1}Y_{i_2} Y_{j_1}Y_{j_2}]-\E [Y_{i_1}Y_{i_2}] \,\E [Y_{j_1}Y_{j_2}]\Big)\\
\le \sum_{\substack{(i_1,i_2),(j_1,j_2)\in I_2(n)\\(i_1,i_2)\neq(j_1,j_2)}} \Big(\P\big(i_1\leftrightarrow \{j_1, j_2\}\big)+\P\big(i_2\leftrightarrow \{j_1, j_2\}\big)\Big).
\end{multline*}
It suffices to treat the two terms individually, and one of them gives
\begin{align*}
\sum_{\substack{(i_1,i_2),(j_1,j_2)\in I_2(n)\\(i_1,i_2)\neq(j_1,j_2)}} \P\big(i_1\leftrightarrow \{j_1, j_2\}\big)&\le \sum_{i_1\in I(n)} \sum_{i_2\sim i_1} \sum_{(j_1,j_2)\in I_2(n)\setminus \{(i_1,i_2)\}}\Big(\P(i_1\leftrightarrow j_1) +\P(i_1\leftrightarrow j_2)\Big)\\
&\le \sum_{i_1\in I(n)} \sum_{i_2\sim i_1}\left( \sum_{\substack{(j_1,j_2)\in I_2(n)\setminus \{(i_1,i_2)\}\\j_1=i_1}} 2 + \sum_{\substack{(j_1,j_2)\in I_2(n)\setminus \{(i_1,i_2)\}\\j_2=i_1}} 2\right.\\
&\quad +\left.\sum_{\substack{(j_1,j_2)\in (I(n)\setminus \{i_1\})_2\\~}}\Big(\P(i_1\leftrightarrow j_1) +\P(i_1\leftrightarrow j_2)\Big) \right)\\
&\le 6 \times 2 \sum_{i_1\in I(n)}\left(5\times 2 + 6 \sum_{j_1\in I(n)\setminus \{i_1\}}\P(i_1\leftrightarrow j_1)\right)\\
&\le 12 n\left(10+6 \frac{1}{\mathrm{e}^{g(\mu^*)}-1}\right)=\mathcal{O}(n),
\end{align*}
where in the second inequality, we separate between cases when $i_1=j_1$ or $i_1=j_2$ and when not, and notice that when they are not equal, then all $(j_1,j_2)$ pairs are disjoint from $i_1$. In the third inequality, we replace the sum for $i_2\sim i_1$ by a factor of~6 (for the triangular lattice), and instead of $(j_1,j_2)$, we sweep for $j_1$, and then for its at most 6 neighbours $j_2$ separately. Thereby we reduced the problem to the previous case and the fourth inequality follows by~\eqref{e:samecomp}. This completes the proofs of Lemma~\ref{l:corr2} and Proposition~\ref{p:second}.
\end{proof}

The proof of Proposition~\ref{p:second} can be adapted to~$\widetilde{Y}$. For example, the following variant of Proposition~\ref{p:means} also holds:
\begin{align}\label{e:convij}
\sup_{\theta\in\Theta}\frac{1}{n_p}\left|\sum_{(i,j)\in I_2}\E_\theta\left[\widetilde{Y}_i\widetilde{Y}_j\right]-\sum_{(i,j)\in I_2}\E_\theta [Y_i Y_j]\right|\ &\conv{n_I\to\infty}\ 0.
\end{align}

\section{Uniform law of large numbers (ULLN) for our process}\label{s:ULLN}

In the interests of conciseness, we continue assuming that there is only one colour: $n_c=1$. This leads to no loss of generality. We prove that the SLLNs, Propositions \ref{p:first} and~\ref{p:second}, hold uniformly over the compact parameter set $\Theta\subset [0,1]\times [0,p_c[$. Similarly to the preceding, we write everything out for~$Y$, but the result is also valid for~$\widetilde{Y}$. 


To prove the uniform version of Proposition~\ref{p:first}, we check that the conditions of the following theorem hold, where we adapted \cite[p.~8, 2~Theorem]{Pollard_1984} or \cite[p.~25, Lemma~3.1]{vandeGeer_2010} to our setting. For the rewriting of the theorem, we exploited that for a sample $\big((U_i)_{i\in L},(V_{ij})_{(i,j)\in L_2}\big)$ of the seeds and edges, any $Y_i$ is nondecreasing in both $\lambda$ and $\mu$. 

\begin{thm}[cf. {\cite[p.~8, 2~Theorem]{Pollard_1984}}, {\cite[p.~25, Lemma~3.1]{vandeGeer_2010}}]\label{th:ULLN}
Suppose that for every $\varepsilon>0$ there exists a finite set of pairs of parameter vectors
\[\mathcal{P}=\left\{\big(\theta^L_r, \theta^U_r\big)\in \big([0,1]\times [0,p_c[\big)^2\ \Big|\ r\in\{1,\dots, N(\varepsilon)\}\right\}\]
such that
\begin{enumerate}
\item for every $r\in\{1,\dots, N(\varepsilon)\}$, the SLLN holds for $\theta^L_r$ and $\theta^U_r$; that is, if $Y$ is generated with parameter value $\theta^L_r$, then
\[\frac{1}{n_I}\left(\sum_{i\in I}Y_i-\sum_{i\in I}\E_{\theta^L_r} Y_i\right)\ \conv{n_I\to\infty}\ 0\]
almost surely, and similarly for~$\theta^U_r$;
\item for every $\theta\in\Theta$, there is an $r\in\{1,\dots, N(\varepsilon)\}$ such that $\theta^L_r\le \theta\le\theta^U_r$ coordinatewise;
\item for every $r\in\{1,\dots, N(\varepsilon)\}$ and $i\in I$, $\E_{\theta^U_r} Y_i - \E_{\theta^L_r} Y_i \le \varepsilon$.
\end{enumerate}
Then the ULLN holds, that is, 
\[\sup_{\theta\in\Theta}\frac{1}{n_I}\left|\sum_{i\in I}Y_i-\sum_{i\in I}\E_\theta Y_i\right|\ \conv{n_I\to\infty}\ 0\]
almost surely, where $Y$ is generated with parameter value~$\theta$.
\end{thm}

We construct $\mathcal{P}$ such that the rectangles $R_r$ spanned by $\theta^L_r=\big(\lambda^L_r,\mu^L_r\big)$ and $\theta^U_r=\big(\lambda^U_r,\mu^U_r\big)$, that is, the closed rectangles $\big\{\big(\lambda^L_r, \mu^L_r\big),\ \big(\lambda^U_r, \mu^L_r\big),\ \big(\lambda^U_r, \mu^U_r\big),\ \big(\lambda^L_r, \mu^U_r\big)\big\}$, cover $\Theta$. By this construction, Condition~2 holds. No matter how we choose finitely many pairs $\big(\theta^L_r, \theta^U_r\big)$, Condition~1 holds for each by Proposition~\ref{p:first}. We achieve Condition~3 by proving Lipschitz continuity of the expectation $\E_\theta Y_i$ in $\theta$.

\begin{lem}\label{l:ucont}
For any $\mu^*\in]0,p_c[$, The expectation $\E_\theta Y_i$ is Lipschitz continuous in $\theta$ over the set $[0,1]\times [0,\mu^*]$ with some Lipschitz constant~$L_0$, which is universal for $i\in L$.
\end{lem}

Lemma~\ref{l:ucont} guarantees uniform continuity in~$\theta$. Instead of $\delta>0$ that corresponds to the $\varepsilon$ required by Theorem~\ref{th:ULLN}, we take a $\delta'\in ]0,\delta[$. For instance, $\delta'=\varepsilon/(2L_0)$ is suitable. We cover $\Theta$ with open rectangles $R_u$ of the above form with diameter $|\theta^U_r-\theta^L_r|=\delta'$. We intersect each $R_u$ with $[0,1]\times [0,p_c[$ to avoid overhangs; they remain relatively open. Because of compactness, there is a finite subcover of~$\Theta$ with such potentially trimmed open rectangles. We define $\mathcal{P}$ via the vertices of these finitely many rectangles. $R_r$ are now closed rectangles with diameter~$\delta'$. They are contained in rectangles with diameter~$\delta$, and Condition~3 is satisfied. In conclusion, a proof of Lemma~\ref{l:ucont} proves the SLLN for~$Y$.

\begin{proof}[Lemma~\ref{l:ucont}]
Consider $\theta, \theta'\in [0,1]\times [0,\mu^*]$, $\theta=(\lambda,\mu)$ and $\theta'=(\lambda',\mu')$. We can assume that $\theta\le\theta'$ coordinatewise. If this were not the case, we would prove the inequality for $\theta^L:=(\lambda\wedge\lambda',\mu\wedge\mu')$ and $\theta^U:=(\lambda\vee\lambda',\mu\vee\mu')$. This suffices since $|\theta-\theta'|=|\theta^L-\theta^U|$, and both $\E_\theta Y_i$ and $\E_{\theta'} Y_i$ are contained in $[\E_{\theta^L} Y_i,\ \E_{\theta^U} Y_i]$ due to monotonicity.

We identify the vertices of~$L$ with the source edges, and fix an ordering of all source and contamination edges: $L\cup L_2=\{e_0,e_1,e_2,\dots\}$. Let $\vartheta:\ \mathbb{N}\to [0,1]$ be such that $\vartheta_k=\lambda$ if $e_k$ is a source edge, and $\vartheta_k=\mu$ if $e_k$ is a contamination edge. Define $\vartheta'$ analogously with $\lambda', \mu'$ in place of $\lambda, \mu$, respectively. Finally, let $\theta^k,\theta'^k:\ \mathbb{N}\to [0,1]$ be defined by
\begin{align*}
(\theta^k)_j&:=\begin{cases}\vartheta'_j&\textrm{if }j<k,\\\vartheta_j&\textrm{if }j\ge k,\end{cases}&(\theta'^k)_j&:=\begin{cases}\vartheta'_j&\textrm{if }j\le k,\\\vartheta_j&\textrm{if }j>k,\end{cases}
\end{align*}
for $k,j\in\mathbb{N}$. Let $\omega_\theta$ be the configuration that is specified by $\big((U_i)_{i\in L},(V_{ij})_{(i,j)\in L_2}\big)$ and parameter~$\theta$ via $(X_i)_{i\in L}$ and $(\xi_{ij})_{(i,j)\in L_2}$. Then
\begin{align*}
\E_{\theta'} Y_i - \E_\theta Y_i&=\P_{\theta'}(Y_i=1)-\P_\theta(Y_i=1)\\
&=\sum_{k=0}^\infty \P\big(Y_i(\omega_{\theta'^k})=1\textrm{ and }Y_i(\omega_{\theta^k})=0\big)\\
&=\sum_{k=0}^\infty (\vartheta'_k-\vartheta_k)\P_{\theta^k}\big(e_k\textrm{ is pivotal for }Y_i=1\big),
\end{align*}
where the second equality is just the law of total probability when we know that $\{Y_i=1\}$ is an increasing event, and the third equality is elaborated in \cite[pp.~41--43]{Grimmett_1999} as such a step is used in the proof of Russo's formula. Note that the concerns in that derivation related to an infinite edge set do not apply here because we have always got only one edge $e_k$ whose parameter differs between $\theta^k$ and~$\theta'^k$. (The price we pay is that each pivotality is with a different parameter vector~$\theta^k$.) If $e_k$ is a source edge, then $\vartheta'_k-\vartheta_k=\lambda'-\lambda$, and if $e_k$ is a contamination edge, then $\vartheta'_k-\vartheta_k=\mu'-\mu$. Further,
\begin{align*}
\P_{\theta^k}\big(e_k\textrm{ is pivotal for }Y_i=1\big)&\le\begin{cases}1,&\textrm{if }e_k\textrm{ is the source edge of vertex }i,\\
\P_{\theta^k}(j\leftrightarrow i),&\textrm{if }e_k\textrm{ is the source edge of vertex }j\neq i,\\
1,&\textrm{if }e_k\textrm{ is an edge incident with }i,\\
\P_{\theta^k}(i_1\leftrightarrow i)+\P_{\theta^k}(i_2\leftrightarrow i),&\textrm{if }e_k\textrm{ is the edge }(i_1,i_2), i\neq i_1,i_2.\end{cases}
\end{align*}
Then
\begin{align*}
\sum_{e_k \textrm{ source edge}}&(\vartheta'_k-\vartheta_k)\P_{\theta^k}\big(e_k\textrm{ is pivotal for }Y_i=1\big)\\
&\le (\lambda'-\lambda)\left(1+\sum_{j\in L\setminus\{i\}}\P_{\theta^k}(j\leftrightarrow i)\right)\\
&\le (\lambda'-\lambda)\left(1+\frac{1}{\mathrm{e}^{g(\mu^*)}-1}\right)
\end{align*}
by~\eqref{e:samecomp}. Using tricks from the proof of Lemma~\ref{l:corr2},
\begin{align*}
\sum_{e_k \textrm{ contamination edge}}&(\vartheta'_k-\vartheta_k)\P_{\theta^k}\big(e_k\textrm{ is pivotal for }Y_i=1\big)\\
&\le (\mu'-\mu)\left(6+\sum_{(i_1,i_2)\in (L\setminus\{i\})_2}\big(\P_{\theta^k}(i_1\leftrightarrow i)+\P_{\theta^k}(i_2\leftrightarrow i)\big)\right)\\
&\le (\mu'-\mu)\left(6+2\times 6\sum_{j\in L\setminus\{i\}}\P_{\theta^k}(j\leftrightarrow i)\right)\\
&\le (\mu'-\mu)\left(6+\frac{12}{\mathrm{e}^{g(\mu^*)}-1}\right).
\end{align*}
Consequently,
\begin{align*}
\E_{\theta'} Y_i - \E_\theta Y_i&\le (\lambda'-\lambda)\left(1+\frac{1}{\mathrm{e}^{g(\mu^*)}-1}\right) + (\mu'-\mu)\left(6+\frac{12}{\mathrm{e}^{g(\mu^*)}-1}\right)\\
&\le (\lambda'-\lambda+\mu'-\mu)\left(6+\frac{12}{\mathrm{e}^{g(\mu^*)}-1}\right)\le L_0|\theta'-\theta|
\end{align*}
for some $L_0>0$ because in finite dimensions, all norms are equivalent.
\end{proof}


Lemma~\ref{l:ucont} for $\E_\theta[Y_i Y_j]$ ($i\sim j$) can be shown by a now straightforward adjustment of the original proof. This then implies that the following modification of Theorem~\ref{th:ULLN} holds.
\begin{thm}\label{th:ULLN2}
Suppose that the conditions of Theorem~\ref{th:ULLN} hold with the following updates to points 1 and~3:
\begin{enumerate}
\item[1'] for every $r\in\{1,\dots, N(\varepsilon)\}$, if $Y$ is generated with parameter value $\theta^L_r$, then
\[\frac{1}{n_p}\left(\sum_{(i,j)\in I_2}Y_i Y_j-\sum_{(i,j)\in I_2}\E_{\theta^L_r} [Y_i Y_j]\right)\ \conv{n_I\to\infty}\ 0\]
almost surely, and similarly for~$\theta^U_r$;
\item[3'] for every $r\in\{1,\dots, N(\varepsilon)\}$ and $(i,j)\in I_2$, $\E_{\theta^U_r} [Y_i Y_j] - \E_{\theta^L_r} [Y_i Y_j]\le \varepsilon$.
\end{enumerate}
Then the ULLN holds, that is
\[\sup_{\theta\in\Theta}\frac{1}{n_p}\left|\sum_{(i,j)\in I_2}Y_i Y_j-\sum_{(i,j)\in I_2}\E_\theta [Y_i Y_j]\right|\ \conv{n_I\to\infty}\ 0\]
almost surely, where $Y$ is generated with parameter value~$\theta$.
\end{thm}

The derivations and results of this section hold with $\widetilde{Y}$, too. The assumption of identifiability, Proposition~\ref{p:means} and \eqref{e:convij} together guarantee that the second term of~\eqref{e:convalpha} converges to zero almost surely if and only if $\theta=\theta_0$. Additionally, the arguments of the first term of~\eqref{e:convalpha} converge uniformly to those of the second term almost surely, due to the conclusions of Theorems \ref{th:ULLN} and~\ref{th:ULLN2}. This proves our main theorem, Theorem~\ref{th:main}.

We followed the philosophy that the dataset $\mathcal{Y}$ comes from the infinite lattice~$L$ although only a finite subset is observed. This is an idealised view that assumes the existence of a process on the infinite lattice. Otherwise, when the dataset is of type $\widetilde{\mathcal{Y}}$, the derivation is simpler because Proposition~\ref{p:means} and \eqref{e:convij} are not needed.



\section{Computer testing of the proposed method}\label{s:comp}

\subsection{Implementation}\label{s:impl}
We implemented the proposed MSM parameter estimator in the \textsc{Matlab} software (The MathWorks, Inc.), and we report our findings in this section. See also \cite{Beck_2015} for an early version with $n_c=3$ colours. For the the objective function
\begin{equation}
\alpha\left(\begin{array}{c} \left(\bar{\mathcal{Y}}^\ell - \frac{1}{n_s}\sum_{s=1}^{n_s} \bar{Y}^{\ell,s} \right)_{\ell\in\{1,\dots,n_c\}}\\ \left(\bar{\mathcal{Z}}^\ell - \frac{1}{n_s}\sum_{s=1}^{n_s} \bar{Z}^{\ell,s} \right)_{\ell\in\{1,\dots,n_c\}}\end{array} \right),\label{e:alpha}
\end{equation}
we chose the quadratic form $\alpha(\eta)=\eta\T \Omega \eta$ the following way:
\begin{align}
\Omega&=\diag\Big((\bar{\mathcal{Y}}^1)^{-2},\dots,(\bar{\mathcal{Y}}^{n_c})^{-2},(\bar{\mathcal{Z}}^1)^{-2},\dots,(\bar{\mathcal{Z}}^{n_c})^{-2}\Big).\label{e:Omega}
\end{align}
In the unlikely case that a $\bar{\mathcal{Y}}^\ell$ or a $\bar{\mathcal{Z}}^\ell$ is zero, the corresponding diagonal element of $\Omega$ is set to~$1$. Through this normalisation, we expect each coordinate to contribute roughly equally to the sum.

Common random numbers are used during the exploration of the parameter space. This removes an element of fluctuation as different $\theta=(\lambda^1,\dots,\lambda^{n_c},\mu)\in\Theta$ are tested. We propose two alternative methods for sampling synthetic datasets. Method~1 is the canonical approach. We draw and fix independent random variables from the uniform distribution on $[0,1]$: $(U^{\ell,s}_i)$ for $\ell\in\{1,\dots,n_c\}$, $s\in\{1,\dots,n_s\}$, $i\in I$, and $(V^s_{ij})$ for $s\in\{1,\dots,n_s\}$, $(i,j)\in I_2$. Thereafter, for each parameter vector, seeding and the open or closed state of edges are defined by
\begin{align*}
X^{\ell,s}_i&:=\begin{cases}1&\textrm{if }U^{\ell,s}_i<\lambda^\ell,\\0&\textrm{otherwise},\end{cases}&&\textrm{for }\ell\in\{1,\dots,n_c\},\ s\in\{1,\dots,n_s\},\ i\in I;\\
\xi^s_{ij}&:=\begin{cases}1&\textrm{if }V^s_{ij}<\mu,\\0&\textrm{otherwise},\end{cases}&&\textrm{for }s\in\{1,\dots,n_s\},\ (i,j)\in I_2.
\end{align*}
This method gives a binomially distributed number of open edges and, similarly, seeded vertices for each colour~$\ell$.

We anticipate that it is beneficial for the parameter estimation to remove the randomness in the numbers of seeds and open edges, and to make exactly as many edges open as their expected number, $\zeta(\mu n_p)$, where $\zeta$ is the rounding to the nearest integer with some tie-breaking rule. The same is stipulated for seeds: $\zeta(\lambda^\ell n_I)$ random vertices shall be seeded with colour~$\ell$. This is what Method~2 does. We see this as a variance-reduction trick that achieves lower variance by introducing dependencies between random draws: for example, by knowing the state of all edges but one, we can infer the state of the remaining edge.

Let $S_n$ denote the set of permutations of $\{1,\dots,n\}$. In Method~2, one draws permutations $(\sigma^{\ell,s})$ from $S_{n_I}$ independently, uniformly at random for $\ell\in\{1,\dots,n_c\}$, $s\in\{1,\dots,n_s\}$, and independent permutations $(\tau^s)$ from $S_{n_p}$ uniformly at random for $s\in\{1,\dots,n_s\}$. With these permutations fixed, for each $\theta\in\Theta$, one lets
\begin{align*}
X^{\ell,s}_i&:=\begin{cases}1&\textrm{if }\sigma^{\ell,s}(i)\le \zeta(\lambda^\ell n_I),\\0&\textrm{otherwise},\end{cases}&&\textrm{for }\ell\in\{1,\dots,n_c\},\ s\in\{1,\dots,n_s\},\ i\in I;\\
\xi^s_{ij}&:=\begin{cases}1&\textrm{if }\tau^s \big((i,j)\big)\le\zeta(\mu n_p),\\0&\textrm{otherwise},\end{cases}&&\textrm{for }s\in\{1,\dots,n_s\},\ (i,j)\in I_2.
\end{align*}

Minimisation over the parameter space $\Theta$ is conducted with the \textsc{Matlab} routine \texttt{fminsearchbnd} \cite{DErrico_2012} for constrained optimisation. $\lambda^\ell_{\textrm{max}}:=\bar{\mathcal{Y}}^\ell$ is certainly an upper bound on what $\lambda^\ell$ any point estimator might estimate ($\ell\in\{1,\dots,n_c\}$) as this is the moment estimate in case $\mu=0$. The upper bound $\mu_{\textrm{max}}$ on $\mu$ is left to the user's judgement.

The last user input in addition to $n_s$, Method and $\mu_{\textrm{max}}$ is $n_{\textrm{opt}}$ which specifies how many different initial states to try in the optimisation runs. We expect an inverse relationship between seeding rates and the contamination rate, given the data. Thus the initial parameter values for $k\in\{1,\dots,n_{\textrm{opt}}\}$ are chosen as
\begin{align*}
\lambda^\ell_{\textrm{initial}}(k)&=\bigg(1-\frac{k-1}{n_{\textrm{opt}}}\bigg)\lambda^\ell_{\textrm{max}},&&\textrm{for }\ell\in\{1,\dots,n_c\},\\
\mu_{\textrm{initial}}(k)&=\chi_{\{n_{\textrm{opt}}>1\}}\frac{k-1}{n_{\textrm{opt}}-1}\, \mu_{\textrm{max}}.&&
\end{align*}

\subsection{Results}

In order to test the performance of the proposed estimation procedure, we created a number of synthetic datasets with $n_c=3$ colours, different sizes and different, known parameter vectors using Method~1. Tables~\ref{tb:table1}--\ref{tb:table4} report the results of estimating $\theta_0=(\lambda^1,\lambda^2,\lambda^3,\mu)$ using different input settings $(n_s, n_{\textrm{opt}}, \mu_{\textrm{max}})$.

The two estimators, which are based on Methods 1 and~2 of random number generation, are denoted by $\hat{\theta}^{(\textrm{M}1)}_{n_s,n_I}$ and $\hat{\theta}^{(\textrm{M}2)}_{n_s,n_I}$, respectively. We display the relative bias of the estimators in percentage terms:
\[d^{(\textrm{M}1)}=100\left| 1-\hat{\theta}^{(\textrm{M}1)}_{n_s,n_I}\Big/\theta_0\right|\]
(the operations are coordinatewise), and analogously, $d^{(\textrm{M}2)}$ for $\hat{\theta}^{(\textrm{M}2)}_{n_s,n_I}$. Finally, we let $\alpha_{\hat{\theta}^{(\textrm{M}1)}_{n_s,n_I}}$ and $\alpha_{\hat{\theta}^{(\textrm{M}2)}_{n_s,n_I}}$ denote the value of the objective function $\alpha$ in~\eqref{e:alpha} at~$\hat{\theta}^{(\textrm{M}1)}_{n_s,n_I}$ and $\hat{\theta}^{(\textrm{M}2)}_{n_s,n_I}$, respectively.

The computations were conducted on a laptop computer equipped with a $2.8$~GHz Intel Core~{i7-2640M} dual-core processor and $8$~GB RAM. Although it is clear that the $n_{\textrm{opt}}$ parameter searches and for each, the $n_s$ simulations lend themselves to parallelisation, our implementation does not benefit from this insight. The columns of $\alpha_{\hat{\theta}^{(\textrm{M}1)}_{n_s,n_I}}$ and $\alpha_{\hat{\theta}^{(\textrm{M}2)}_{n_s,n_I}}$ display in brackets running times in seconds for completing the parameter estimation procedure. These times are indicative only and their use for comparisons is limited as less demanding other tasks were also running on the computer simultaneously. As far as we can tell, the parameter estimation ran in RAM without resorting to swap memory on disk.

We found no definitive answer as to whether Method 1 or~2 is preferable. Table~\ref{tb:table3} suggests Method~2, but Table~\ref{tb:table4} is as inconclusive as smaller-sized datasets.

Broadly, the relative bias of the estimates becomes smaller as $n_I$ grows. From $n_I=25\times 25=625$ to $n_I=500\times 500=250000$, the relative bias of the $\mu$ estimate improves from about $35$--$40\%$ to below~$5\%$. We have also observed that as $n_I$ grows, there is ever less need to try several initial states because the solutions tend to converge to the same estimator. In our experience, the existence of local optima that necessitate a greater $n_{\textrm{opt}}$ were characteristic of the smaller lattice sizes only.

In the smallest dataset, Table~\ref{tb:table1}, one can observe that $\lambda^1$ and $\lambda^2$ are consistently overestimated, whereas $\lambda^3$ and $\mu$ are underestimated in all six estimations. This turned out to be due to a quirk of the randomly generated dataset. While $(\lambda^1,\lambda^2,\lambda^3)=(0.1,0.05,0.07)$,
in reality, the dataset had
\[\bar{\mathcal{X}}=\frac{1}{n_I}\left(\sum_{i\in I}\mathcal{X}^1_i,\,\sum_{i\in I}\mathcal{X}^2_i,\,\sum_{i\in I}\mathcal{X}^3_i\right)=(0.1072,\,0.0528,\,0.0640).\]
One can notice that in Table~\ref{tb:table1}, $\lambda^1$ and $\lambda^2$ are overestimated to a greater extent than how much $\lambda^3$ is underestimated. Then the observed systemic underestimation of~$\mu$ is consistent with this in light of the expected inverse relationship between seeding and contamination described at the end of Section~\ref{s:impl}.

In Table~\ref{tb:table3}, where the lattice size $n_I=300\times 300=90000$ is most relevant to our practical application in Section~\ref{s:exp}, $(n_s, n_{\textrm{opt}})\in\{(2,8),(4,4),(8,2)\}$ allow a comparison of different input choices with approximately identical computational cost. $(n_s, n_{\textrm{opt}})=(4,4)$ and right behind it $(8,2)$ proved to be the best choices, beating $(n_s, n_{\textrm{opt}})=(2,8)$. Against the expectations, $(n_s, n_{\textrm{opt}})=(5,5)$ happened to not improve the estimate with input $(4,4)$. On this lattice size, $\mu$ is estimated to $10\%$ accuracy with $1$--$2$ hours running time. In Table~\ref{tb:table4}, we get better than $5\%$ accuracy on a larger lattice with $7$--$8$ hours running time.

In Tables \ref{tb:table1}--\ref{tb:table4}, for fixed $n_I$, $\alpha_{\hat{\theta}^{(\textrm{M}1)}_{n_s,n_I}}$ and $\alpha_{\hat{\theta}^{(\textrm{M}2)}_{n_s,n_I}}$ tend to decrease for increasing $n_s$. This is reassuring, although not a necessity because it is possible that the synthetic dataset is atypical and more simulations (higher $n_s$) do not make it easier to approximate it. Instead, overfitting might yield the lowest $\alpha$ values.


\begin{table}
\begin{center}
\begin{tabular}{ >{$}c<{$} >{$}c<{$} >{$}c<{$} >{$}c<{$} | >{$}r<{$} >{$}r<{$} >{$}r<{$} >{$}r<{$} >{$}r<{$} >{$}r<{$}}
n_s & n_{\textrm{opt}} & \mu_{\textrm{max}} & \theta_0 & \hat{\theta}^{(\textrm{M}1)}_{n_s,n_I} & d^{(\textrm{M}1)} & \alpha_{\hat{\theta}^{(\textrm{M}1)}_{n_s,n_I}} & \hat{\theta}^{(\textrm{M}2)}_{n_s,n_I} & d^{(\textrm{M}2)} & \alpha_{\hat{\theta}^{(\textrm{M}2)}_{n_s,n_I}}\\ \hline
10&10&0.1&0.1 & 0.1292 & 29.16\% & 0.0124 & 0.1277 & 27.74\% & 0.0215\\
&&& 0.05 & 0.0595 & 19.00\% & & 0.0637 & 27.39\% & \\
&&& 0.07 & 0.0611 & 12.71\% & & 0.0538 & 23.11\% & \\
&&& 0.06 & 0.0433 & 27.86\% & (107\,\textrm{s})&0.0376 & 37.28\% & (94\,\textrm{s})\\
\hline
50&10&0.1&0.1 & 0.1289 & 28.92\% & 0.0128 & 0.1267 & 26.75\% & 0.00875\\
&&& 0.05 & 0.0641 & 28.26\% & & 0.0611 & 22.18\% & \\
&&& 0.07 & 0.0603 & 13.79\% & & 0.0588 & 15.94\% & \\
&&& 0.06 & 0.0394 & 34.37\% & (523\,\textrm{s})&0.0422 & 29.65\% & (467\,\textrm{s})\\
\hline
100&10&0.1&0.1 & 0.1350 & 35.01\% & 0.0108 & 0.1259 & 25.87\% & 0.0107\\
&&& 0.05 & 0.0631 & 26.18\% & & 0.0623 & 24.53\% & \\
&&& 0.07 & 0.0613 & 12.44\% & & 0.0595 & 15.00\% & \\
&&& 0.06 & 0.0360 & 40.00\% & (1.03\mathrm{e+}03\,\textrm{s})&0.0403 & 32.84\% & (922\,\textrm{s})
\end{tabular}
\caption{Six estimates for a synthetic dataset with $n_I=25\times 25=625$ vertices ($n_p=1776$) and $\theta_0=(0.1, 0.05, 0.07, 0.06)$. In this synthetic dataset, the relative frequency of the incidence of the three colours in the seeding is $\bar{\mathcal{X}}=(0.107, 0.0528, 0.064)$, while in the contamination-impacted observed data, it is $\bar{\mathcal{Y}}=(0.15, 0.0752, 0.08)$. The relative frequency of adjacent vertices having an open edge between them is $0.0574$.}
\label{tb:table1}
\end{center}
\end{table}

\begin{table}
\begin{center}
\begin{tabular}{ >{$}c<{$} >{$}c<{$} >{$}c<{$} >{$}c<{$} | >{$}r<{$} >{$}r<{$} >{$}r<{$} >{$}r<{$} >{$}r<{$} >{$}r<{$}}
n_s & n_{\textrm{opt}} & \mu_{\textrm{max}} & \theta_0 & \hat{\theta}^{(\textrm{M}1)}_{n_s,n_I} & d^{(\textrm{M}1)} & \alpha_{\hat{\theta}^{(\textrm{M}1)}_{n_s,n_I}} & \hat{\theta}^{(\textrm{M}2)}_{n_s,n_I} & d^{(\textrm{M}2)} & \alpha_{\hat{\theta}^{(\textrm{M}2)}_{n_s,n_I}}\\ \hline
20&10&0.05&0.07 & 0.0734 & 4.79\% & 1.55\mathrm{e-}05 & 0.0739 & 5.52\% & 0.00044\\
&&& 0.05 & 0.0508 & 1.52\% & & 0.0492 & 1.66\% & \\
&&& 0.04 & 0.0390 & 2.49\% & & 0.0392 & 2.07\% & \\
&&& 0.03 & 0.0259 & 13.77\% & (3.08\mathrm{e+}03\,\textrm{s})&0.0262 & 12.79\% & (3.01\mathrm{e+}03\,\textrm{s})\\
\hline
40&10&0.05&0.07 & 0.0737 & 5.28\% & 5.96\mathrm{e-}06 & 0.0733 & 4.78\% & 0.000108\\
&&& 0.05 & 0.0505 & 0.93\% & & 0.0500 & 0.06\% & \\
&&& 0.04 & 0.0392 & 2.05\% & & 0.0393 & 1.82\% & \\
&&& 0.03 & 0.0253 & 15.54\% & (5.9\mathrm{e+}03\,\textrm{s})&0.0252 & 16.00\% & (5.63\mathrm{e+}03\,\textrm{s})
\end{tabular}
\caption{Four estimates for a synthetic dataset with $n_I=100\times 100=10000$ vertices ($n_p=29601$) and $\theta_0=(0.07, 0.05, 0.04, 0.03)$. In this synthetic dataset, the relative frequency of the incidence of the three colours in the seeding is $\bar{\mathcal{X}}=(0.0717, 0.0497, 0.0387)$, while in the contamination-impacted observed data, it is $\bar{\mathcal{Y}}=(0.085, 0.0585, 0.0455)$. The relative frequency of adjacent vertices having an open edge between them is $0.0303$.}
\label{tb:table2}
\end{center}
\end{table}

\begin{table}
\begin{center}
\begin{tabular}{ >{$}c<{$} >{$}c<{$} >{$}c<{$} >{$}c<{$} | >{$}r<{$} >{$}r<{$} >{$}r<{$} >{$}r<{$} >{$}r<{$} >{$}r<{$}}
n_s & n_{\textrm{opt}} & \mu_{\textrm{max}} & \theta_0 & \hat{\theta}^{(\textrm{M}1)}_{n_s,n_I} & d^{(\textrm{M}1)} & \alpha_{\hat{\theta}^{(\textrm{M}1)}_{n_s,n_I}} & \hat{\theta}^{(\textrm{M}2)}_{n_s,n_I} & d^{(\textrm{M}2)} & \alpha_{\hat{\theta}^{(\textrm{M}2)}_{n_s,n_I}}\\ \hline
2&8&0.05&0.05 & 0.0472 & 5.57\% & 0.00241 & 0.0488 & 2.43\% & 0.000149\\
&&& 0.06 & 0.0572 & 4.71\% & & 0.0593 & 1.18\% & \\
&&& 0.03 & 0.0308 & 2.71\% & & 0.0298 & 0.51\% & \\
&&& 0.02 & 0.0225 & 12.59\% & (4\mathrm{e+}03\,\textrm{s})&0.0220 & 9.86\% & (3.95\mathrm{e+}03\,\textrm{s})\\
\hline
4&4&0.05&0.05 & 0.0480 & 4.03\% & 0.000968 & 0.0489 & 2.12\% & 6.11\mathrm{e-}05\\
&&& 0.06 & 0.0578 & 3.59\% & & 0.0589 & 1.90\% & \\
&&& 0.03 & 0.0305 & 1.65\% & & 0.0298 & 0.83\% & \\
&&& 0.02 & 0.0223 & 11.64\% & (3.8\mathrm{e+}03\,\textrm{s})&0.0213 & 6.45\% & (3.19\mathrm{e+}03\,\textrm{s})\\
\hline
8&2&0.05&0.05 & 0.0483 & 3.32\% & 0.000904 & 0.0481 & 3.75\% & 0.000612\\
&&& 0.06 & 0.0586 & 2.35\% & & 0.0586 & 2.26\% & \\
&&& 0.03 & 0.0299 & 0.40\% & & 0.0301 & 0.45\% & \\
&&& 0.02 & 0.0223 & 11.60\% & (3.53\mathrm{e+}03\,\textrm{s})&0.0220 & 10.11\% & (3.83\mathrm{e+}03\,\textrm{s})\\
\hline
5&5&0.05&0.05 & 0.0481 & 3.81\% & 0.000964 & 0.0485 & 2.97\% & 0.000495\\
&&& 0.06 & 0.0580 & 3.29\% & & 0.0588 & 1.99\% & \\
&&& 0.03 & 0.0302 & 0.79\% & & 0.0303 & 1.04\% & \\
&&& 0.02 & 0.0221 & 10.69\% & (6.08\mathrm{e+}03\,\textrm{s})&0.0219 & 9.72\% & (6.12\mathrm{e+}03\,\textrm{s})
\end{tabular}
\caption{Eight estimates for a synthetic dataset with $n_I=300\times 300=90000$ vertices ($n_p=268801$) and $\theta_0=(0.05, 0.06, 0.03, 0.02)$. In this synthetic dataset, the relative frequency of the incidence of the three colours in the seeding is $\bar{\mathcal{X}}=(0.0498, 0.0595, 0.0299)$, while in the contamination-impacted observed data, it is $\bar{\mathcal{Y}}=(0.0558, 0.0667, 0.034)$. The relative frequency of adjacent vertices having an open edge between them is $0.0198$.}
\label{tb:table3}
\end{center}
\end{table}

\begin{table}
\begin{center}
\begin{tabular}{ >{$}c<{$} >{$}c<{$} >{$}c<{$} >{$}c<{$} | >{$}r<{$} >{$}r<{$} >{$}r<{$} >{$}r<{$} >{$}r<{$} >{$}r<{$}}
n_s & n_{\textrm{opt}} & \mu_{\textrm{max}} & \theta_0 & \hat{\theta}^{(\textrm{M}1)}_{n_s,n_I} & d^{(\textrm{M}1)} & \alpha_{\hat{\theta}^{(\textrm{M}1)}_{n_s,n_I}} & \hat{\theta}^{(\textrm{M}2)}_{n_s,n_I} & d^{(\textrm{M}2)} & \alpha_{\hat{\theta}^{(\textrm{M}2)}_{n_s,n_I}}\\ \hline
1&1&0.04&0.03 & 0.0336 & 11.86\% & 0.0243 & 0.0331 & 10.47\% & 0.025\\
&&& 0.04 & 0.0444 & 10.96\% & & 0.0450 & 12.59\% & \\
&&& 0.05 & 0.0553 & 10.70\% & & 0.0555 & 10.94\% & \\
&&& 0.02 & 0.0141 & 29.27\% & (355\,\textrm{s})&0.0136 & 31.91\% & (340\,\textrm{s})\\
\hline
5&5&0.04&0.03 & 0.0297 & 1.08\% & 0.000894 & 0.0297 & 1.04\% & 0.00106\\
&&& 0.04 & 0.0394 & 1.43\% & & 0.0395 & 1.13\% & \\
&&& 0.05 & 0.0518 & 3.57\% & & 0.0519 & 3.86\% & \\
&&& 0.02 & 0.0196 & 1.94\% & (2.8\mathrm{e+}04\,\textrm{s})&0.0194 & 3.01\% & (2.48\mathrm{e+}04\,\textrm{s})
\end{tabular}
\caption{Four estimates for a synthetic dataset with $n_I=500\times 500=250000$ vertices ($n_p=748001$) and $\theta_0=(0.03, 0.04, 0.05, 0.02)$. In this synthetic dataset, the relative frequency of the incidence of the three colours in the seeding is $\bar{\mathcal{X}}=(0.0299, 0.0402, 0.0503)$, while in the contamination-impacted observed data, it is $\bar{\mathcal{Y}}=(0.0336, 0.0451, 0.057)$. The relative frequency of adjacent vertices having an open edge between them is $0.01996$.}
\label{tb:table4}
\end{center}
\end{table}

For further analysis, we introduce two more symbols. One might consider a trivial estimator which assumes no contamination occurring: $\hat{\mu}=0$, $\hat{\lambda}^\ell=\bar{\mathcal{Y}}^\ell$. The corresponding $\alpha_{\textrm{triv}}$ denotes a realisation of~$\alpha$ with parameters from this trivial estimator, computed from $n_s$ simulations with Method 1 or~2. $\alpha_{\theta_0}$ denotes a realisation of~$\alpha$ with the true parameter~$\theta_0$ and $n_s$ simulations.


Table~\ref{tb:table5} compares $\alpha_{\hat{\theta}^{(\textrm{M}1)}_{n_s,n_I}}$ and $\alpha_{\hat{\theta}^{(\textrm{M}2)}_{n_s,n_I}}$, $\alpha_{\textrm{triv}}$ and $\alpha_{\theta_0}$ for the four computer-generated datasets of Tables~\ref{tb:table1}--\ref{tb:table4}. Except for the smallest case, $n_I=625$, $\alpha_{\theta_0}$ is always smaller than $\alpha_{\textrm{triv}}$, as expected. Whereas $\alpha_{\theta_0}$ decreases with increasing~$n_I$, $\alpha_{\textrm{triv}}$ stays roughly constant. $\alpha_{\hat{\theta}^{(\textrm{M}1)}_{n_s,n_I}}$ and $\alpha_{\hat{\theta}^{(\textrm{M}2)}_{n_s,n_I}}$ decrease only initially as $n_I$ increases. One would expect them to be between $\alpha_{\theta_0}$ and $\alpha_{\textrm{triv}}$, which tends to hold for larger lattice sizes. In reality, their value is much lower than $\alpha_{\theta_0}$, but the ratio becomes ever less extreme as $n_I$ grows. This is indicative of initially very strong, but later ever less pronounced overfitting.

\begin{table}
\begin{center}
\begin{tabular}{>{$}c<{$} >{$}c<{$} >{$}c<{$} >{$}c<{$} | >{$}c<{$}   >{$}c<{$}  >{$}c<{$}  >{$}c<{$}  >{$}c<{$}  >{$}c<{$}} 
 n_I & n_p & n_s & n_{\textrm{opt}} & \alpha_{\theta_0}^{(\textrm{M}1)} & \alpha_{\theta_0}^{(\textrm{M}2)} & \alpha_{\textrm{triv}}^{(\textrm{M}1)} & \alpha_{\textrm{triv}}^{(\textrm{M}2)} & \alpha_{\hat{\theta}^{(\textrm{M}1)}_{n_s,n_I}} & \alpha_{\hat{\theta}^{(\textrm{M}2)}_{n_s,n_I}} \\
\hline
25\times25 & 1776 & 10 & 10 & 1.06 & 1.77 & 0.6 & 0.555 & 0.0124 & 0.0215\\
25\times25 & 1776 & 50 & 10 & 0.82 & 0.919 & 0.61 & 0.618 & 0.0128 & 0.00875\\
25\times25 & 1776 & 100 & 10 & 1.19 & 1.09 & 0.597 & 0.589 & 0.0108 & 0.0107\\
\hline
100\times100 & 29601 & 20 & 10 & 0.042 & 0.0454 & 0.59 & 0.61 & 1.55\mathrm{e-}05 & 0.00044\\
100\times100 & 29601 & 40 & 10 & 0.0496 & 0.0576 & 0.598 & 0.601 & 5.96\mathrm{e-}06 & 0.000108\\
\hline
300\times300 & 268801 & 2 & 8 & 0.0179 & 0.0048 & 0.646 & 0.604 & 0.00241 & 0.000149\\
300\times300 & 268801 & 4 & 4 & 0.00948 & 0.0014 & 0.653 & 0.624 & 0.000968 & 6.11\mathrm{e-}05\\
300\times300 & 268801 & 8 & 2 & 0.00792 & 0.00659 & 0.654 & 0.614 & 0.000904 & 0.000612\\
300\times300 & 268801 & 5 & 5 & 0.0104 & 0.00786 & 0.651 & 0.622 & 0.000964 & 0.000495\\
\hline
500\times500 & 748001 & 1 & 1 & 0.0223 & 0.0119 & 0.628 & 0.621 & 0.0243 & 0.025\\
500\times500 & 748001 & 5 & 5 & 0.00626 & 0.00791 & 0.633 & 0.624 & 0.000894 & 0.00106
\end{tabular}
\caption{A comparison of the values of the objective functions for the true value~$\theta_0$, for the trivial estimator and for the MSM estimator. The four synthetic datasets used are the same as in Tables~\ref{tb:table1}--\ref{tb:table4}.}\label{tb:table5}
\end{center}
\end{table}

To test the behaviour of the objective function $\alpha_{\theta_0}$ as $n_I\to\infty$, we generated fresh synthetic datasets of different sizes with a common $\theta_0=(0.03,0.04,0.05,0.02)$. Just generating the single dataset of size $1000\times1000$ took $42$ seconds. For this exercise, the single datasets were compared to simulations with common simulation count $n_s=10$. Table~\ref{tb:table6} shows that both $\alpha_{\theta_0}$ and $\widetilde{\alpha}(\eta)=\eta\T \eta$ converge to zero, although $\alpha_{\theta_0}$ has larger values because of the normalisation by $\Omega$ in~\eqref{e:Omega}. This is numerical evidence in support of Propositions \ref{p:first} and~\ref{p:second}, even with fixed~$n_s$.


\begin{table}
\begin{center}
\begin{tabular}{>{$}c<{$} >{$}c<{$} >{$}c<{$} >{$}c<{$} | >{$}c<{$}  >{$}c<{$}  >{$}c<{$}  >{$}c<{$}}
\textrm{Size} & n_I & n_p & n_s & \alpha_{\theta_0}^{(\textrm{M}1)} & \alpha_{\theta_0}^{(\textrm{M}2)} & \widetilde{\alpha}_{\theta_0}^{(\textrm{M}1)} & \widetilde{\alpha}_{\theta_0}^{(\textrm{M}2)} \\
\hline
25\times25 & 625 & 1776 & 10 & 0.202 & 0.16 & 6.65\mathrm{e-}05 & 4.21\mathrm{e-}05\\
100\times100 & 10000 & 29601 & 10 & 0.0427 & 0.0277 & 7.29\mathrm{e-}06 & 2.57\mathrm{e-}06\\
300\times300 & 90000 & 268801 & 10 & 0.00624 & 0.00807 & 1.68\mathrm{e-}06 & 1.57\mathrm{e-}06\\
500\times500 & 250000 & 748001 & 10 & 0.000799 & 0.0015 & 3.87\mathrm{e-}07 & 3.97\mathrm{e-}07\\
707\times707 & 499849 & 1496720 & 10 & 0.000521 & 0.000365 & 5.08\mathrm{e-}07 & 3.99\mathrm{e-}07\\
1000\times1000 & 1000000 & 2996001 & 10 & 0.00127 & 0.00117 & 9.78\mathrm{e-}08 & 9.04\mathrm{e-}08\\
\end{tabular}
\caption{Realisations of the objective function~$\alpha$ for the true parameter value~$\theta_0$ for different synthetic dataset sizes and of the not normalised variant of the objective function $\widetilde{\alpha}(\eta)=\eta\T \eta$. Here $\theta_0=(0.03,0.04,0.05,0.02)$ across fresh synthetic datasets.}\label{tb:table6}
\end{center}
\end{table}


\section{Cross-contamination rate estimation for digital PCR in lab-on-a-chip microfluidic devices}\label{s:exp}

Our motivation for investigating this problem is the need for quality control in parallelised biochemical experiments run in novel, lab-on-a-chip microfluidic devices for applications in basic research, biotechnology, medical diagnostics and rapid vaccine development. Our collaborators Dr~G\"unter Roth and his group (Centre for Biological Systems Analysis [ZBSA], University of Freiburg) develop such microfluidic devices. The central element of their system is a rectangular well plate with 15~mm edge lengths, with more than 100,000 wells of 19~p$\ell$ volume each. The wells on this chip are arranged in a hexagonal tiling pattern (honeycomb lattice).

Whereas the rival microfluidic technology uses an emulsion of water droplets flowing in an oil medium, this array-based setup fixes a spatial structure, allowing the otherwise neglected analysis of cross-contamination between reaction volumes. Our focus is on evaluating an experiment particularly well suited for this purpose, whose results generalise to other experiments conducted in this lab-on-a-chip device.

In the \emph{digital PCR} experiment, a solution of DNA samples is injected onto the well plate, at such a low concentration that most wells receive 0 or 1 DNA molecule (hence the name \emph{digital}). In the particular case, the solution is a mixture of three different DNA species. We call these template molecules \emph{seeds}. The well plate is covered with a lid (a microscope slide) that is pre-coated with covalently bound DNA primers~\cite{Hoffmann_etal_2012a}. The well plate together with the lid serve to insulate the reaction volumes from each other. The DNA templates are amplified in each of the wells independently with a polymerase chain reaction (PCR). In more detail, the template molecules hybridise to the surface-bound primers and the PCR elongates these primers to form the comp\-le\-men\-ta\-ry strand of the template. In the next heating step, the templates become resolved, whereas the generated comp\-le\-men\-ta\-ry DNA strands stay covalently bound to the surface. The single-strand templates will bind to other surface-bound primers and turn them too into comp\-le\-men\-ta\-ry strands via polymerisation. The result of the PCR cycles is that the whole glass surface above the well gets covered with immobilised comp\-le\-men\-ta\-ry DNA strands. They mirror the spatial arrangement of the initial seed pattern of the wells.

After the PCR, the three comp\-le\-men\-ta\-ry DNA species on the slide are identified via three specifically binding fluorescent hybridisation probes (fluorophores) and their presence or absence can be determined by imaging~\cite{Hoffmann_etal_2012}. In the fluorescent image of the slide (Figure~\ref{f:chips}), we see either black background (where there was no seed), spots in one of the three primary colours indicating a single seed, and sometimes a mixture of two or three primary colours indicating heterogeneous seeding by multiple seeds. Sometimes we also see clusters of one colour, or an unusually high number of mixed colours, indicating cross-contamination between adjacent wells. This happens when the lid is not fitted tightly and during thermal cycling, liquid exchange occurs between reaction volumes around trapped air bubbles and dust particles. In the readout it remains unclear if two neighbours with the same colour (or a single well with a mix of two colours, which has coloured neighbours) were initiated by two seeds or one contaminated the other (Fig.~\ref{f:chips}, bottom panel).

\begin{figure}[ht]
\centering
\includegraphics[width=8cm]{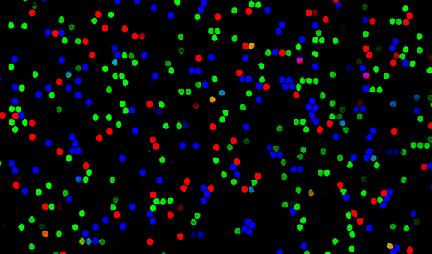}\\
\includegraphics[width=8cm]{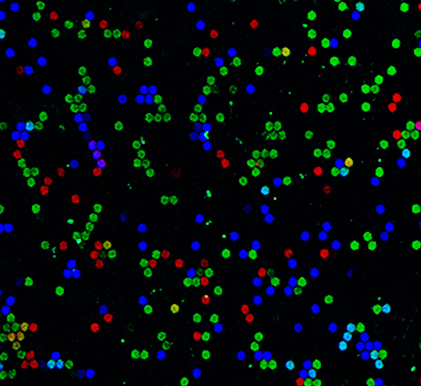}
\caption{(\emph{top})~Image of a glass slide from a digital PCR experiment with little sign of cross-contamination. \cite{Rath_2014} (\emph{bottom})~Image of a slide with clustering fluorescent signals and a higher prevalence of cyan and yellow colours, suggesting higher cross-contamination rate.}\label{f:chips}
\end{figure}

For cross-contamination rate estimation for this experimental setup it is necessary to define a mathematical model of the physical process. It has to involve the triangular lattice, which is the dual of the hexagonal tiling, and colouring of its vertices. The total numbers of DNA templates of each type $\ell\in\{1,\dots,n_c\}$ present in the chip are likely well approximated by $n_c$ discretised normal random variables. We can safely assume that each well receives a Poisson distributed random number of DNA templates of type~$\ell$ because then due to the superposition property, the total number of type $\ell$ templates in the chip is also Poisson distributed, which is close to a normal distribution. The Bernoulli distributed $(X^\ell_i)$ used in our model for seeding are really just a proxy to the either zero or positive value of the corresponding Poisson distribution. From a value $\lambda^\ell$ of the Bernoulli parameter, we can infer the parameter $\widetilde{\lambda}^\ell$ of the respective Poisson distribution through the identity $\lambda^\ell=1-\mathrm{e}^{-\widetilde{\lambda}^\ell}$.

It is also natural to model the possibility of contamination by open edges. It is a useful shortcut to draw the state of the edges independently of the seeding so that an open edge means only the possibility of propagation, which is contingent on the presence of seeds.
There are modelling choices to be made. Contamination might be
\begin{description}
\item[(i)] unidirectional (there is the possibility of a pair of independent, oppositely oriented directed edges $\xi_{i\to j}$ and $\xi_{j\to i}$ between any two adjacent vertices $i\sim j$), or
\item[(ii)] symmetric (undirected edges $\xi_{ij}$).
\end{description}
Open edges might be best represented by
\begin{description}
\item[(1)] independent Bernoulli variables, or by
\item[(2)] locally correlated $0$--$1$ random variables.
\end{description}
Contamination might be
\begin{description}
\item[(A)] confined to neighbours, or
\item[(B)] it might propagate via a series of open edges.
\end{description}

The choice of (ii,1,B) yields the model put forward in Section~\ref{s:intro} (Figure~\ref{f:chips_simul}). Its strength is that it can use standard percolation theory. Our MSM estimator was developed for this model.

For the quality certification of this lab-on-a-chip device, it is useful to estimate in addition to~$\mu$, the total number of vertices which belong to a non-trivial component of the percolation graph. These vertices are the wells which were not insulated from their neighbours. Beyond the digital PCR paradigm, in experimental setups where most wells are expected to give some signal, vertices that are connected to any other are likely to give false signals.

An easy upper bound results from noticing that each edge turns at most two additional vertices connected. For small values of~$\mu$, edges are actually unlikely to share endpoints. The number of edges is distributed according to a binomial distribution with parameters $n_p$ and~$\mu$. Therefore the mean number of potentially contaminated vertices can be estimated as
\[\E\left[\sum_{|C|\ge 2}|C|\right]\le 2\mu n_p \sim 6\mu n_I\]
where the asymptotic equality holds under the assumption that the boundary of~$I$ is `small'. For concrete examples, the conversion from~$n_p$ to~$n_I$ can be accurately determined.

Another approach results by noticing
\begin{align*}
\E\left[\sum_{|C|\ge 2}|C|\right]&=\E\left[n_I-\sum_{i\in I}\chi_{\{|C(i)|=1\}}\right]\\
&=n_I-n_I (1-\mu)^6 +e\\
&=\left(6\mu - 15\mu^2 +\sum_{k=3}^6 \binom{6}{k} (-1)^{k+1} \mu^k \right)n_I +e,
\end{align*}
where $e$ is the correction for boundary vertices.


\begin{figure}[h]
\centering
\includegraphics[width=\textwidth]{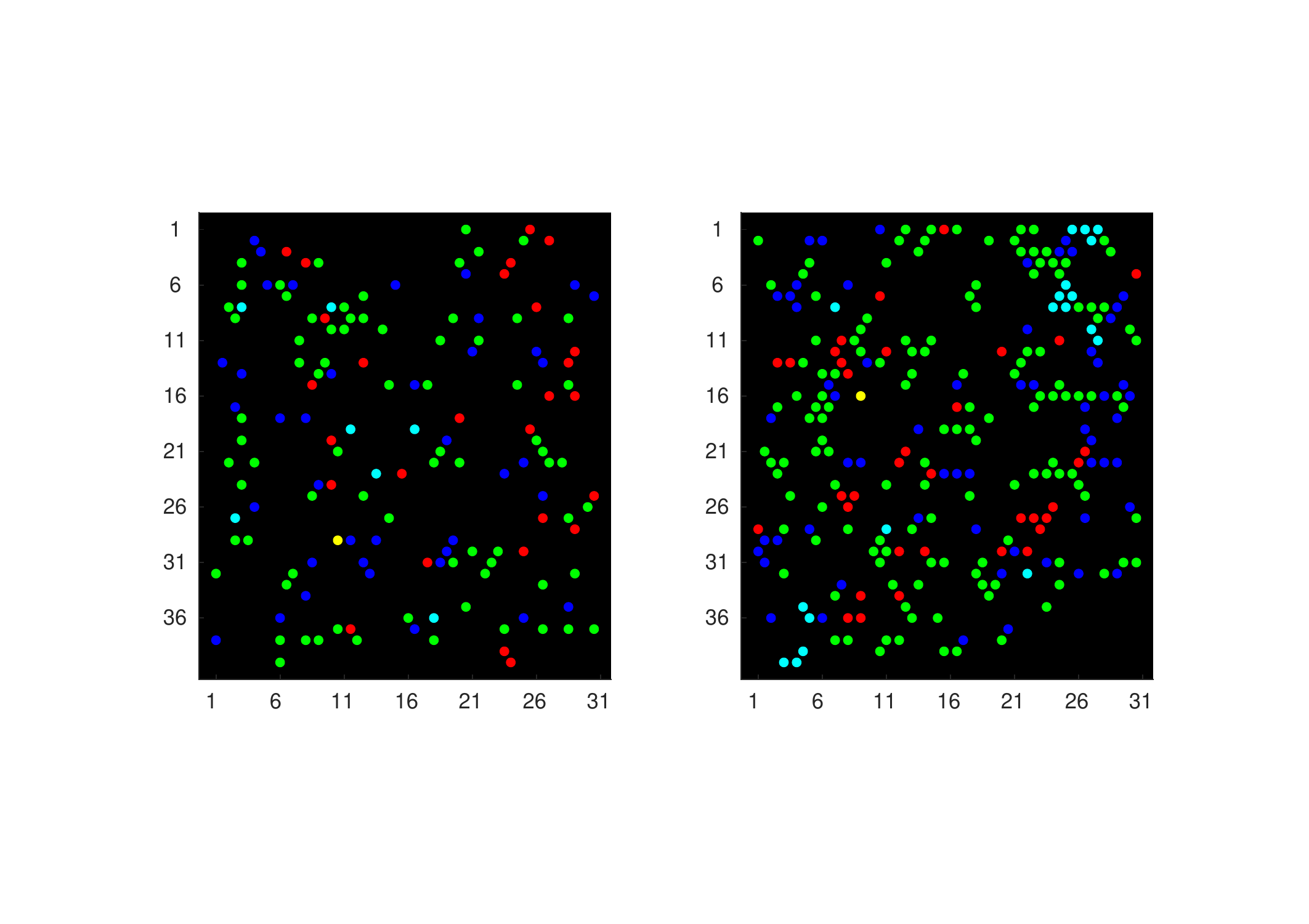}
\caption{(\emph{left})~Computer simulation of a glass slide from a digital PCR experiment under model~(ii,1,B) with $\theta_0=(\lambda^{\textrm{red}},\lambda^{\textrm{green}},\lambda^{\textrm{blue}},\mu)=(0.02,0.07,0.05,0.01)$ and relatively little sign of cross-contamination. (\emph{right})~Computer-simulated slide with clustering fluorescent signals and a higher prevalence of cyan colour, suggesting higher cross-contamination rate. Here $\theta_0=(0.02,0.07,0.05,0.06)$.}\label{f:chips_simul}
\end{figure}

Simpler cases are given by (i,1,A) and (ii,1,A) where the moments $\E[Y_i^\ell]$, $\E[Y_i^\ell Y_i^m]$ and $\E[Y_i^\ell Y_j^\ell]$ ($\ell\neq m$, $i\sim j$) can be computed explicitly. We used \textsc{Mathematica} (Wolfram Research, Inc.) to deal with the many terms, and we report truncations of the complete result for space considerations in the case (i,1,A). It is anticipated in the practical application that $\mu<\lambda^\ell$ for every $\ell$. For non-boundary vertices, under this assumption on the anticipated magnitudes, the dominant terms of the moments of interest in decreasing order are given as
\begin{align*}
\E[Y_i^\ell]&=\P(X_i^\ell=1)+\P(X_i^\ell=0)\sum_{k=1}^6\binom{6}{k}\mu^k(1-\mu)^{6-k}\left(1-(1-\lambda^\ell)^k\right)\\
&=\lambda^\ell+6\lambda^\ell \mu-6(\lambda^\ell)^2\mu-15(\lambda^\ell)^2\mu^2+\mathcal{O}\big((\lambda^\ell)^5\big),\\
\E[Y_i^\ell Y_i^m]&=\P(X_i^\ell X_i^m=1)+\P(X_i^\ell=1,X_i^m=0)\sum_{k=1}^6\binom{6}{k}\mu^k(1-\mu)^{6-k}\left(1-(1-\lambda^m)^k\right)\\
&\quad +\P(X_i^\ell=0,X_i^m=1)\sum_{k=1}^6\binom{6}{k}\mu^k(1-\mu)^{6-k}\left(1-(1-\lambda^\ell)^k\right)\\
&\quad +\P(X_i^\ell=X_i^m=0)\sum_{k=1}^6\binom{6}{k}\mu^k(1-\mu)^{6-k}\left(1-(1-\lambda^\ell)^k\right)\left(1-(1-\lambda^m)^k\right)\\
&=\lambda^\ell \lambda^m +18 \lambda^\ell \lambda^m\mu - 12\Big((\lambda^\ell)^2\lambda^m +\lambda^\ell (\lambda^m)^2\Big)\mu +30\lambda^\ell \lambda^m \mu^2+\mathcal{O}\big(\max\{\lambda^\ell,\lambda^m\}^5\big).
\end{align*}
For $\E[Y_i^\ell Y_j^\ell]$ ($i\sim j$), in the case $X_i^\ell + X_j^\ell=1$, the empty vertex might have been contaminated by the seeded vertex, or it might have been contaminated from its five remaining neighbours. If $X_i^\ell =X_j^\ell=0$, then one can separate cases according to the seeding status of the two shared neighbours of $i$ and~$j$. These considerations give
\begin{align*}
\E[Y_i^\ell Y_j^\ell]&=(\lambda^\ell)^2 + 2\lambda^\ell \mu+8(\lambda^\ell)^2 \mu + 2\lambda^\ell \mu^2 -10(\lambda^\ell)^3 \mu + 9(\lambda^\ell)^2 \mu^2 +\mathcal{O}\big((\lambda^\ell)^5\big).
\end{align*}

These $n_c^2/2 + 3 n_c/2$ moment equations provide the opportunity to estimate the $n_c+1$ parameters via the method of moments. Of these, it is $\E[Y_i^\ell Y_j^\ell]$ where the first term with~$\mu$ is highest up in the magnitude ranking, underpinning the physical intuition that the cooccurrence of a colour in two adjacent vertices is the most informative moment about the contamination rate~$\mu$.

Notably, the model (ii,1,A) gives exactly the above moment equations if for any $(i,j)\in I_2$,
\begin{align*}
&\P(\xi_{i\to j}=1)=\P(\xi_{j\to i}=1)=\mu&&\textrm{in model (i,1,A), and}\\
&\P(\xi_{ij}=1)=\mu&&\textrm{in model (ii,1,A).}
\end{align*}
The reason is that the propagation of colours is limited to neighbours, so already second neighbours are ruled out. An edge between $i$ and~$j$ makes a difference in any of the above three moments if and only if $X_i^\ell + X_j^\ell=1$. Say, $X_j^\ell=1=1-X_i^\ell$. Then $\xi_{j\to i}$ has the same effect on these moments as $\xi_{ij}$, and also the same probability because one can marginalise over the state of~$\xi_{i\to j}$. However, $\E[Y_i^\ell Y_j^\ell Y_i^m Y_j^m]$ would differ between the models (i,1,A) and~(ii,1,A). See also the Appendix of~\cite{Frisch_Hammersley_1963}.


\section{Discussion and open problems}

This paper describes the solution of a statistical problem motivated by a concrete practical need. The mathematical modelling part is solved in one of multiple possible ways, and the choice of (ii,1,B) brings in bond percolation into the statistical model. The percolation is subcritical. The parameter estimation method we propose is the MSM, which gives a point estimate. We prove that it is strongly consistent in the limit as the sample size $n_I$ tends to infinity. It is an important point that the number of simulations per proposed parameter vector, $n_s$, can remain bounded to achieve this result.

What is unusual in our setting is that although the sample size is large, it is not independent (nor identically distributed). Introductory percolation theory is used to upper bound long-range dependencies between the $n_I$ samples.

We have implemented the method and its accuracy is tested on synthetic datasets in practically relevant parameter ranges. Estimates for wetlab data are to be published by our collaborators G\"unter Roth and his co-workers in the microfluidics literature.

Parameter estimation in connection with a (static) percolation model is not common in the literature, apart from the quest for the critical value. Dynamic percolation models and dynamic random graphs on a fixed vertex set provide a framework for the contact network in modelling the spread of epidemics. Gilligan and Gibson have been particularly active in studying statistical problems for spatiotemporal models of plant epidemic spread \cite{Gibson_etal_2006, Ludlam_etal_2011_RSIF}. Gilligan and co-workers also conducted experiments with the fungal pathogen \emph{Rhizoctonia solani} grown in a Petri dish to test how infection probability between a pair of lattice points (that is, the parameter $\mu$ of percolation in the directed case~(i)) depends on their distance and how invasive spread (percolation) probability depends on nutrient availability in lattice points and on the distance between lattice points \cite{Bailey_Otten_Gilligan_2000}. They also demonstrated that the random removal (blocking) of sites can hinder and even stop disease spread by driving it subcritical \cite{Otten_Bailey_Gilligan_2004}.

Beyond the almost sure convergence and the numerical studies with synthetic data, we cannot predict the accuracy of our estimator for instance in terms of confidence intervals. It is known that under regularity conditions, especially that the estimator is continuously differentiable with respect to the parameter~$\theta$, $\sqrt{n_I}(\hat{\theta}_{n_s,n_I}-\theta_0)$ is asymptotically normal with known limiting variance \cite[Section~2.3.1]{Gourieroux_Monfort_1996}. It is also possible to choose $\Omega$ optimally, that is, to minimise this asymptotic variance \cite[Section~2.3.4]{Gourieroux_Monfort_1996}. However, our estimator is not even continuous in~$\theta$ because we use what is called a frequency simulator. It is unknown to us whether it is possible to replace the frequency simulator with some importance sampling to achieve asymptotic normality.


Maximum likelihood estimation (MLE) would have the advantage over MSM that its output is reproducible. Its computational cost might also be lower. Consider the following. We know that black areas have no seeds but we have no information about contamination (edges) in them. We also know that at boundaries between different colours, there is no open edge. Therefore, for a MLE, one needs to establish the probabilities of patches with a fixed colour without knowing which vertices were seeded and which got contaminated only.

We wonder if it is possible by using a generating function that encodes the probabilities of seeding and open edges to compute the total probability that the particular patch was created: each vertex in a patch has been seeded or contaminated from a seed somewhere within the patch. We were only able to derive this generating function for patches that are a linear chain of vertices.

General finite, connected patch shapes (subgraphs) are called \emph{(lattice)} \emph{animals}. Bousquet-M\'elou did much work on characterising them via generating functions \cite{Bousquet-Melou_1998, Bousquet-Melou_Rechnitzer_2002}. Our patches can arise as a disjoint union of adjacent connected components (animals). For our application, it would suffice to develop a recursion which allows one to compute generating functions of small patches (large patches are rare) with a computer algebra system. The difficulty is that the problem is two dimensional, and a patch must be split in all possible ways into two disjoint parts in the recursion. Any newly added vertex might have been seeded, or contaminated from the rest of the patch, but it might have itself contaminated other empty vertices of the patch.

Notably, the MSM estimator can be turned into an Approximate Bayesian computation (ABC) method very easily. One needs to fix a prior distribution on $\Theta$ and a small $\varepsilon >0$. The ABC rejection algorithm draws finitely many independent $\theta\in\Theta$ parameter values from the prior distribution. The objective function~\eqref{e:alpha} is evaluated for each proposed $\theta$. The simulations used for the evaluation should no longer use common random numbers but independent ones, and $n_s$ can be set to one. If the value of the objective function is less than $\varepsilon$, then the proposed $\theta$ is accepted, otherwise it is rejected. This way the set of accepted $\theta$ is a good approximation of the posterior distribution.

We have not yet tested model fit due to the lack of experimental data. As contamination is caused by the imperfect fit of the glass lid and trapped bubbles and dust, we anticipate that locally positively correlated open edges might be needed in the model. That is, case (ii,2,B) deserves close attention. One way of modelling positive correlations is to apply the Ising model to the edges. Let $\widetilde{\xi}_{ij}=2\xi_{ij}-1\in\{-1,+1\}$. Then the energy or the Hamiltonian function of a configuration $\xi$ of open edges is
\[H(\xi)=-J\sum_{i<j<k}( \widetilde{\xi}_{ij}\widetilde{\xi}_{ik} + \widetilde{\xi}_{ij}\widetilde{\xi}_{jk} + \widetilde{\xi}_{ik}\widetilde{\xi}_{jk} ) - \widetilde{\mu}\sum_{(i,j)\in I_2}\widetilde{\xi}_{ij}\]
for some $J>0$ and $\widetilde{\mu}<0$, and in the first sum, out of the three terms those are missing where an adjacency condition is not met: $\widetilde{\xi}_{ij}=0$ if $i\nsim j$, so that every pair of incident edges appears once. The probability of the system being in state $\xi$ is proportional to $\mathrm{e}^{-\beta H(\xi)}$ for some~$\beta>0$. Although we have two new parameters $J$ and the inverse temperature $\beta$ in addition to $\widetilde{\mu}$, the increase in degrees of freedom is really just one, $\beta J$ and $\beta\widetilde{\mu}$ relative to~$\mu$.






\section{Acknowledgements}

The authors are grateful to G\"unter Roth and Christin Rath (ZBSA, University of Freiburg) for proposing the problem, for their relentless help in clarifying details of the experimental protocol and for providing sample images. The authors also thank Robin Ryder (Paris Dauphine University) for suggesting the method of simulated moments and Ed Crane (University of Bristol) and Peter Pfaffelhuber (University of Freiburg) for insights. 
B.~M. thanks the AXA Research Fund for their financial support in the form of a postdoctoral fellowship, and the Isaac Newton Institute for Mathematical Sciences (Cambridge, UK) for support and hospitality during the programme \emph{Stochastic dynamical systems in biology: numerical methods and applications} when work on this paper was undertaken. This work was thereby supported by EPSRC Grant Number EP/K032208/1.

\appendix

\renewcommand\thefigure{\thesection.\arabic{figure}}
\setcounter{figure}{0}

\section{Identifiability and numerical estimates of the selected moments}\label{s:ident}

We outline why we conjecture that the parameter $\theta=(\lambda^1,\dots,\lambda^{n_c},\mu)$ is identifiable from the moments $\left((\E Y^\ell_i)_{\ell\in\{1,2,\dots,n_c\}}, (\E [Y^\ell_i Y^\ell_j])_{\ell\in\{1,2,\dots,n_c\}}\right)$ ($i\sim j$). If we focus on just one colour~$\ell$, then the graph of the function $(\lambda^\ell,\mu)\mapsto \E Y^\ell_i$ on the domain $[0,1]\times [0,p_c]$ has level curves which go from high~$\lambda^\ell$ and low~$\mu$ to low~$\lambda^\ell$ and high~$\mu$. In words, the density $\E Y^\ell_i$ of colour $\ell$ is constant if we compensate for a decreasing seeding rate $\lambda^\ell$ by an appropriately increasing contamination rate~$\mu$. The function $(\lambda^\ell,\mu)\mapsto \E [Y^\ell_i Y^\ell_j]$ ($i\sim j$) has level curves with the same property.

However, we conjecture that the level curves of $\E Y^\ell_i$ and $\E [Y^\ell_i Y^\ell_j]$ do not coincide, instead they intersect. While either one of the two moments narrows down the possible value of the parameter vector to one of its level curves, the two moments jointly specify the intersection point of two level curves, which uniquely identifies the parameter value~$(\lambda^\ell,\mu)$.

We provide numerical evidence to back up this claim. For $n_c=1$, we sampled $\E Y_i$ and $\E [Y_i Y_j]$ in $142$ logarithmically spaced parameter vectors. We made an exception to the logarithmic rule to additionally sample along the line of critical~$\mu$ (Figure~\ref{f:cover}). Dataset~A contains a broader coverage of $100$ parameter vectors. For each of these, we generated independently $n_s=5$ realisations of the process on a lattice~$I'$ of size $300\times 300$, and took its central $100\times 100$ sublattice $I\subset I'$ as our data. $\E Y_i$ and $\E [Y_i Y_j]$ are estimated as averages over the central sublattice over $n_s=5$ realisations.

In Dataset~B, $56$ parameter vectors are considered which have lower $\lambda$ values in comparison with Dataset~A, save for an overlap of $14$ parameter vectors. For each vector, we generated independently $n_s=5$ realisations of the process on a lattice~$I'$ of size $1500\times 1500$, and its central $1000\times 1000$ sublattice $I\subset I'$ serves as our data.

The sublattice sizes were selected such that in both datasets, the mean number of seeds is at least~$5$ in the central sublattice used for sampling, even for their respective lowest $\lambda$ values ($\lambda=5\times 10^{-4}$ in Dataset~A, and approximately $5.23\times 10^{-6}$ in Dataset~B). At the larger lattice size used for Dataset~B, for $\mu$ values larger than what we tested, the step of finding the connected open components to generate the data became prohibitively time consuming.

\begin{figure}[ht]
\centering
\includegraphics[width=10cm]{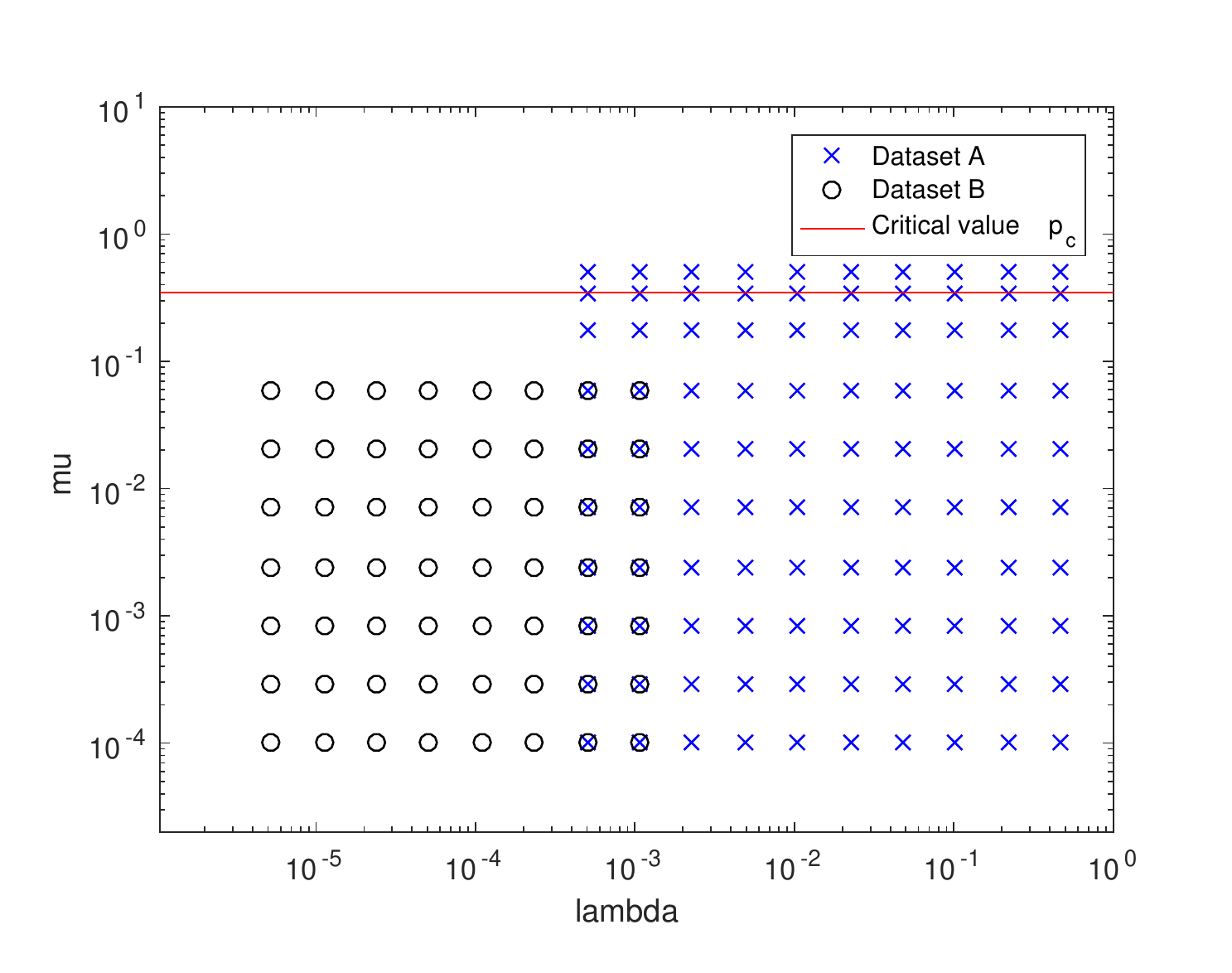}
\caption{Sampled parameter values $\theta=(\lambda,\mu)$. Dataset~A spans $[5\times 10^{-4},0.4676]\times [10^{-4},0.5]$ and Dataset~B spans $[5.23\times 10^{-6}, 0.00107]\times [10^{-4}, 0.0595]$.}\label{f:cover}
\end{figure}

Figures~\ref{f:AYi}--\ref{f:BY} display graphs and level curves of the two coordinates of
\[(\lambda,\mu)\mapsto \left(\frac{1}{n_s n_I}\sum_{s=1}^{n_s} \sum_{i\in I}Y^s_i,\ \frac{1}{n_s n_p}\sum_{s=1}^{n_s} \sum_{(i,j)\in I_2}Y^s_i Y^s_j\right).\]
Close observation of the level curves seems to show that those in Figure~\ref{f:AYi} fan out with different slopes from a smaller region, while those in Figure~\ref{f:AYiYj} are closer to parallel. This supports our conjecture that level curves of one type intersect level curves of the other type in exactly one point, giving identifiability, except perhaps for a null set or otherwise small subset of~$\Theta$ where the two types of level curves coincide.


\begin{figure}[ht]
\centering
\includegraphics[height=12cm]{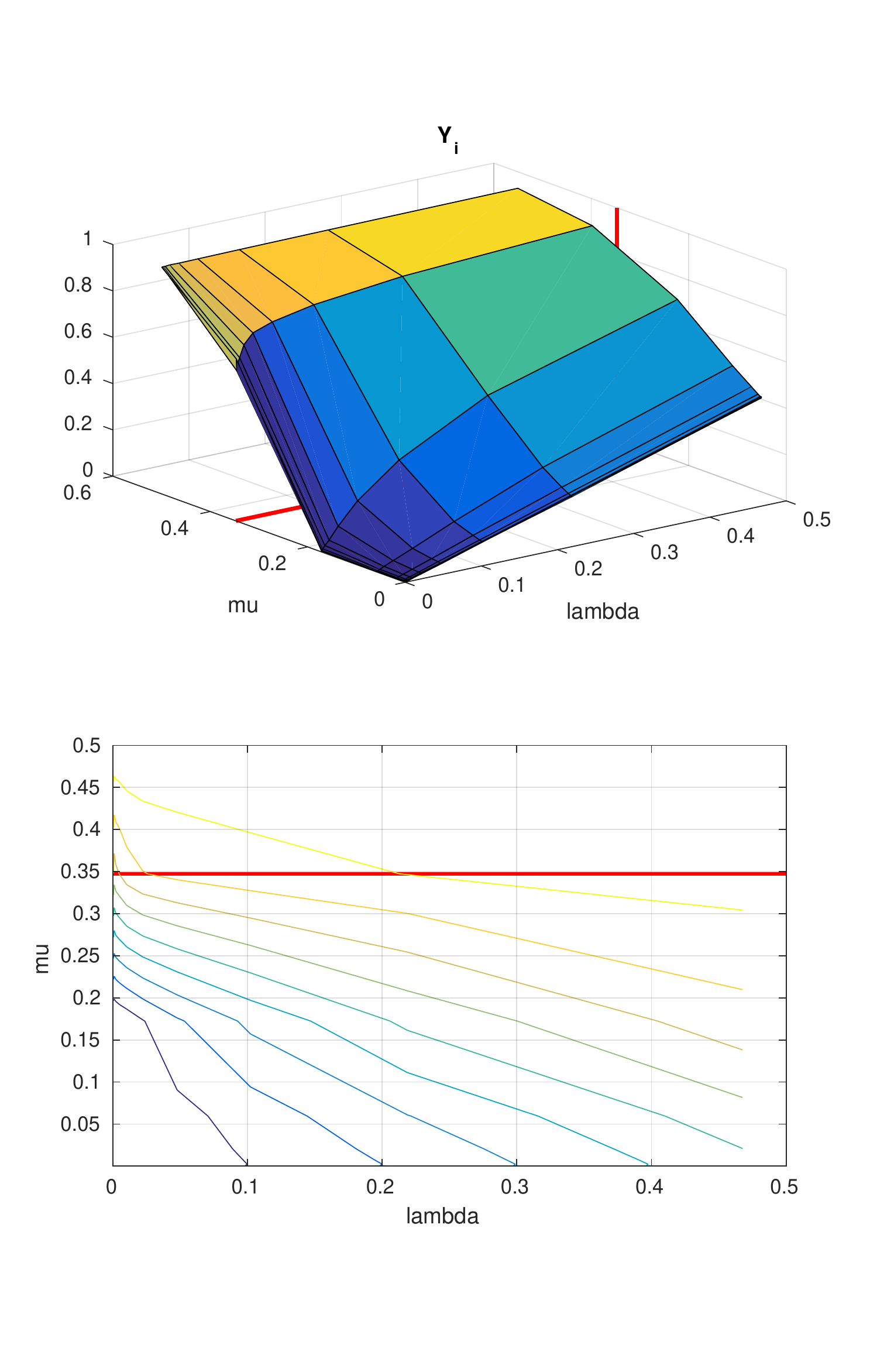}
\caption{(\emph{top})~Empirical means of $Y_i$ for the various parameter vectors of Dataset~A. (\emph{bottom})~Level curves of this function. The red lines mark the critical value~$p_c$.}\label{f:AYi}
\end{figure}

\begin{figure}[ht]
\centering
\includegraphics[height=12cm]{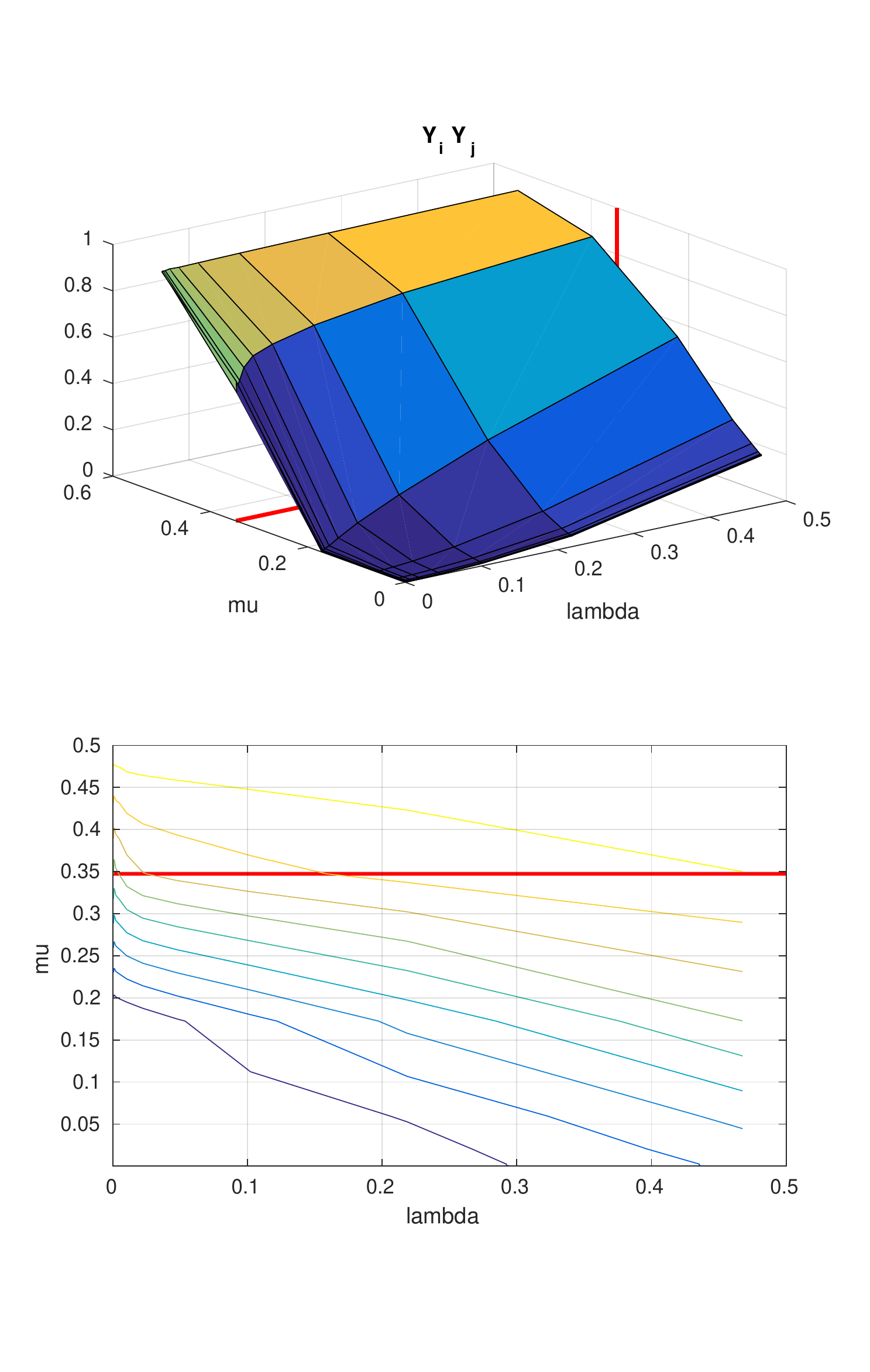}
\caption{(\emph{top})~Empirical means of $Y_i Y_j$ ($i\sim j$) for the various parameter vectors of Dataset~A. (\emph{bottom})~Level curves of this function. The red lines mark the critical value~$p_c$.}\label{f:AYiYj}
\end{figure}

\begin{figure}[hb]
\centering
\includegraphics[height=12cm]{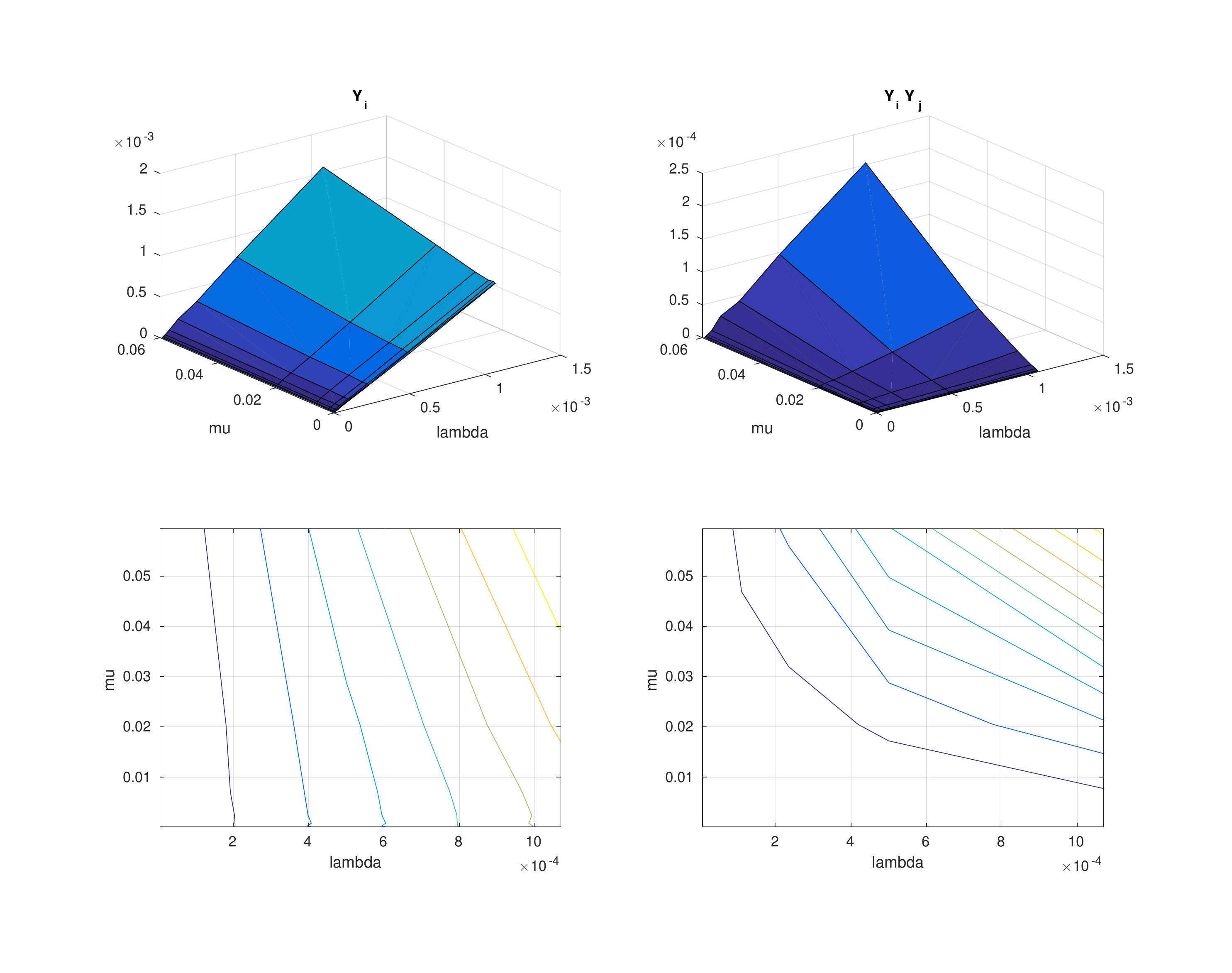}
\caption{(\emph{top left})~Empirical means of $Y_i$ for the various parameter vectors of Dataset~B. (\emph{bottom left})~Level curves of this function. (\emph{top right})~Empirical means of $Y_i Y_j$ ($i\sim j$) for the various parameter vectors of Dataset~B. (\emph{bottom right})~Level curves of this function.}\label{f:BY}
\end{figure}

\bibliographystyle{plainnat}


\end{document}